\documentclass[11pt]{article}
\usepackage{graphicx}
\usepackage{amsmath}
\usepackage{amsfonts}
\usepackage{amssymb}
\usepackage{amsthm}
\usepackage[T1]{fontenc}
\usepackage{url}
\usepackage{color}
\usepackage{defs}
\usepackage{fonts}
\usepackage{esint}
\usepackage{dsfont}
\usepackage[margin=1in]{geometry}
\usepackage{amssymb,bm}
\usepackage{amsmath}
\usepackage{amsthm}
\usepackage{graphicx}
\usepackage[active]{srcltx} 
\usepackage{hyperref}
\hypersetup{pdfborder=0 0 0}

\newtheorem{theorem}{{\sc Theorem}}[section]
\newtheorem{proposition}[theorem]{{\sc Proposition}}

\newtheorem{lemma}[theorem]{{\sc Lemma}}
\newtheorem{corollary}[theorem]{Corollary}

\def\XXint#1#2#3{{\setbox0=\hbox{$#1{#2#3}{\int}$ }
\vcenter{\hbox{$#2#3$ }}\kern-.6\wd0}}

\bmdefine\BGa{\alpha}
\bmdefine\BGb{\beta}
\bmdefine\BGd{\delta}
\bmdefine\BGe{\epsilon}
\bmdefine\BGve{\varepsilon}
\bmdefine\BGf{\phi}
\bmdefine\BGvf{\varphi}
\bmdefine\BGg{\gamma}
\bmdefine\BGc{\chi}
\bmdefine\BGi{\iota}
\bmdefine\BGk{\kappa}
\bmdefine\BGl{\lambda}
\bmdefine\BGn{\eta}
\bmdefine\BGm{\mu}
\bmdefine\BGv{\nu}
\bmdefine\BGp{\pi}
\bmdefine\BGth{\theta}
\bmdefine\BGvth{\vartheta}
\bmdefine\BGr{\rho}
\bmdefine\BGvr{\varrho}
\bmdefine\BGs{\sigma}
\bmdefine\BGvs{\varsigma}
\bmdefine\BGt{\tau}
\bmdefine\BGj{\tau}
\bmdefine\BGu{\upsilon}
\bmdefine\BGo{\omega}
\bmdefine\BGx{\xi}
\bmdefine\BGy{\psi}
\bmdefine\BGz{\zeta}
\bmdefine\BGD{\Delta}
\bmdefine\BGF{\Phi}
\bmdefine\BGG{\Gamma}
\bmdefine\BGL{\Lambda}
\bmdefine\BGP{\Pi}
\bmdefine\BGT{\Theta}
\bmdefine\BGS{\Sigma}
\bmdefine\BGU{\Upsilon}
\bmdefine\BGO{\Omega}
\bmdefine\BGX{\Xi}
\bmdefine\BGY{\Psi}

\bmdefine\BCA{{\mathcal A}}
\bmdefine\BCB{{\mathcal B}}
\bmdefine\BCC{{\mathcal C}}
\bmdefine\BCD{{\mathcal D}}
\bmdefine\BCE{{\mathcal E}}
\bmdefine\BCF{{\mathcal F}}
\bmdefine\BCG{{\mathcal G}}
\bmdefine\BCH{{\mathcal H}}
\bmdefine\BCI{{\mathcal I}}
\bmdefine\BCJ{{\mathcal J}}
\bmdefine\BCK{{\mathcal K}}
\bmdefine\BCL{{\mathcal L}}
\bmdefine\BCM{{\mathcal M}}
\bmdefine\BCN{{\mathcal N}}
\bmdefine\BCO{{\mathcal O}}
\bmdefine\BCP{{\mathcal P}}
\bmdefine\BCQ{{\mathcal Q}}
\bmdefine\BCR{{\mathcal R}}
\bmdefine\BCS{{\mathcal S}}
\bmdefine\BCT{{\mathcal T}}
\bmdefine\BCU{{\mathcal U}}
\bmdefine\BCV{{\mathcal V}}
\bmdefine\BCW{{\mathcal W}}
\bmdefine\BCX{{\mathcal X}}
\bmdefine\BCY{{\mathcal Y}}
\bmdefine\BCZ{{\mathcal Z}}

\bmdefine\Bzr{ 0}
\bmdefine\Ba{ a}
\bmdefine\Bb{ b}
\bmdefine\Bc{ c}
\bmdefine\Bd{ d}
\bmdefine\Be{ e}
\bmdefine\Bf{ f}
\bmdefine\Bg{ g}
\bmdefine\Bh{ h}
\bmdefine\Bi{ i}
\bmdefine\Bj{ j}
\bmdefine\Bk{ k}
\bmdefine\Bl{ l}
\bmdefine\Bm{ m}
\bmdefine\Bn{ n}
\bmdefine\Bo{ o}
\bmdefine\Bp{ p}
\bmdefine\Bq{ q}
\bmdefine\Br{ r}
\bmdefine\Bs{ s}
\bmdefine\Bt{ t}
\bmdefine\Bu{ u}
\bmdefine\Bv{ v}
\bmdefine\Bw{ w}
\bmdefine\Bx{ x}
\bmdefine\By{ y}
\bmdefine\Bz{ z}
\bmdefine\BA{ A}
\bmdefine\BB{ B}
\bmdefine\BC{ C}
\bmdefine\BD{ D}
\bmdefine\BE{ E}
\bmdefine\BF{ F}
\bmdefine\BG{ G}
\bmdefine\BH{ H}
\bmdefine\BI{ I}
\bmdefine\BJ{ J}
\bmdefine\BK{ K}
\bmdefine\BL{ L}
\bmdefine\BM{ M}
\bmdefine\BN{ N}
\bmdefine\BO{ O}
\bmdefine\BP{ P}
\bmdefine\BQ{ Q}
\bmdefine\BR{ R}
\bmdefine\BS{ S}
\bmdefine\BT{ T}
\bmdefine\BU{ U}
\bmdefine\BV{ V}
\bmdefine\BW{ W}
\bmdefine\BX{ X}
\bmdefine\BY{ Y}
\bmdefine\BZ{ Z}

\usepackage{accents}


\usepackage{wrapfig}
\usepackage{tikz}
\usepackage{xcolor}

\numberwithin{equation}{section}

\begin{document}

\title{
Consistency in fractional Korn's inequalities on Lipschitz domains
}
\author{Tadele Mengesha\footnote{Department of Mathematics, University of Tennessee, Knoxville, TN 37996, USA. Email: \texttt{mengesha@utk.edu}} \ \and 
James M. Scott\footnote{Department of Mathematics and Statistics, Auburn University, Auburn, AL 36849, USA. Email: \texttt{james.m.scott@auburn.edu}}
}
\date{\today}
\maketitle

\begin{abstract}
    We investigate fractional analogs of Korn's inequality on bounded Lipschitz domains in several contexts. Constrained and unconstrained inequalities are established; constraints include quotient spaces and spaces incorporating various classes of boundary conditions or volume constraints. We also expand the parameters of integrability for which the inequalities hold. Specifically, a fractional Korn's first unconstrained inequality holds for vector fields in the fractional Sobolev space $W^{s,p}(\Omega;\bbR^n)$ for all $p \in [1,\infty]$ and all $s \in (0,1)$.
    In the cases $p = 1$ and $p = \infty$, the inequality has no local analog. We quantify this by investigating the consistency of the inequalities as the differentiability parameter $s \to 1$; the universal constant in the inequality remains stable for $p \in (1,\infty)$ but degenerates for $p \in \{1,\infty\}$.
    We introduce an explicit $L^p$-bounded projection operator onto infinitesimal rigid motions as a key tool to quantitatively strengthen the inequality.
\end{abstract}


\section{Introduction}\label{sec:intro}
In this paper, we establish fractional analogs of Korn's first and second inequalities in arbitrary bounded Lipschitz domains. 
Classically, Korn's inequalities bound the $L^p$-norm of the full gradient tensor $\nabla \bu$ of a Sobolev vector field $\bu: \Omega \to \mathbb{R}^n$ by the $L^p$-norm of the symmetric part of its gradient tensor, 
\begin{equation*}
    e(\bu)(\bx) := \frac{1}{2} \left(\nabla \bu(\bx) + \nabla \bu^\top(\bx)\right),
\end{equation*}
together with the $L^p$-norm of the vector field itself.

For $1 < p < \infty$, Korn's first (unconstrained) inequality \cite{Korn1,Korn2,Kond-Oleinik} states that there exists a constant $C = C(n,p,\Omega) > 0$ such that
\begin{equation}\label{eq:Intro-Korn1stCase2}
    \inf_{\bspsi \in \cR} \|\bu - \bspsi \|_{W^{1,p}(\Omega)} \leq C \|e(\bu)\|_{L^p(\Omega)} \quad \text{for all } \bu \in W^{1,p}(\Omega;\mathbb{R}^n),
\end{equation}
where 
$\cR$ denotes the space of infinitesimal rigid motions
\begin{equation*}
    \mathcal{R} := \{ \bspsi(\bx) = \mathbf{A}\bx + \mathbf{b} : \bA \in \bbR^{n \times n}, \, \mathbf{b} \in \mathbb{R}^n, \, \mathbf{A} + \bA^T = \bzero \}.
\end{equation*}
Consider any continuous projection map $\Pi$ from $L^p(\Omega;\bbR^n)$ onto $\cR$ (see \cite{Ciarlet2025} for the existence of a similar map); a consequence of \eqref{eq:Intro-Korn1stCase2} is
\begin{equation}\label{eq:Intro-Korn1stCase2:Proj}
    \vnorm{\bu - \Pi[\bu]}_{W^{1,p}(\Omega)}
    \leq C(n,p,\Omega) \|e(\bu)\|_{L^p(\Omega)} \quad \text{for all } \bu \in W^{1,p}(\Omega;\mathbb{R}^n).
\end{equation}
Using the continuity of $\Pi$, 
\eqref{eq:Intro-Korn1stCase2:Proj} leads to Korn's second inequality: there exists a constant $C = C(n,p,\Omega) > 0$ such that
\begin{equation}\label{eq:Intro-Korn2nd}
    \|\nabla \bu\|_{L^p(\Omega)} \leq C \left( \|e(\bu)\|_{L^p(\Omega)} + \|\bu\|_{L^p(\Omega)} \right) \quad \text{for all } \bu \in \mathcal{X}^{1,p}(\Omega),
\end{equation}
where $\mathcal{X}^{1,p}(\Omega) := \{ \bu \in L^p(\Omega;\mathbb{R}^n) : e(\bu) \in L^p(\Omega;\mathbb{R}^{n \times n}) \}$.
Inequality \eqref{eq:Intro-Korn2nd} establishes the equality $\mathcal{X}^{1,p}(\Omega) = W^{1,p}(\Omega;\mathbb{R}^n)$.
From \eqref{eq:Intro-Korn2nd}, a standard compactness argument yields Korn's first (constrained) inequality:
\begin{equation}\label{eq:Intro-Korn1stCase1}
    \|\bu\|_{W^{1,p}(\Omega)} \leq C \|e(\bu)\|_{L^p(\Omega)} \quad \text{for all } \bu \in \mathring{W}^{1,p}(\Omega;\mathbb{R}^n),
\end{equation}
where $\mathring{W}^{1,p}(\Omega;\mathbb{R}^n)$ is any closed subspace of $W^{1,p}(\Omega;\mathbb{R}^n)$ that has trivial intersection with $\cR$.
These classical inequalities play a fundamental role in continuum mechanics, linearized elasticity, and fluid dynamics \cite{Ciarlet,Kond-Oleinik,Temam-Miran}.
For details on the equivalence of \eqref{eq:Intro-Korn1stCase2}-\eqref{eq:Intro-Korn2nd} see \cite{Ciarlet2025}.


A ``fractional analog'' of these inequalities refers to estimates of this type for vector fields in the fractional Sobolev space $W^{s,p}(\Omega;\mathbb{R}^n)$ for $s \in (0,1)$. A vector field $\bu \in L^p(\Omega;\mathbb{R}^n)$ belongs to $W^{s,p}(\Omega;\mathbb{R}^n)$, $p < \infty$, if and only if its Gagliardo seminorm is finite:
\begin{equation*}
    [\bu]_{W^{s,p}(\Omega;\mathbb{R}^n)}^p = \int_{\Omega}\int_{\Omega} \frac{|\bu(\bx) - \bu(\by)|^p}{|\bx - \by|^{n+ps}} \, \mathrm{d}\by \, \mathrm{d}\bx < \infty. 
\end{equation*}
The space $W^{s,p}(\Omega;\mathbb{R}^n)$ is a Banach space equipped with the norm $\|\bu\|_{W^{s,p}(\Omega)}^p = \|\bu\|_{L^p(\Omega)}^p + [\bu]_{W^{s,p}(\Omega)}^p$. 
Locally, for a smooth vector field the seminorm behaves as the integrable function $|\bx-\by|^{-d-sp+p}$ multiplied by the quantity
\begin{equation*}
    \frac{\bu(\by)-\bu(\bx)}{|\by-\bx|} = \grad \bu(\bx) \cdot \frac{\by-\bx}{|\by-\bx|} + \cO(|\by-\bx|) \quad \text{ as } \by \to \bx.
\end{equation*}
Indeed, the rescaled seminorm $(1-s)^{\frac{1}{p}} [\bu]_{W^{s,p}(\Omega)}$ converges as $s \to 1$ to the $W^{1,p}$ seminorm $\vnorm{\grad \bu}_{L^p(\Omega)}$; see \cite{BourgainBrezisMironescu2001}.

An analogous characterization of $e(\bu)$ arises by projecting the difference quotient $\frac{\bu(\by) - \bu(\bx)}{\vert{}\by - \bx\vert{}}$ along unit vector $\frac{\by - \bx}{\vert{}\by - \bx\vert{}}$. Specifically, Taylor expansion yields
\begin{equation*}
\frac{\bu(\by) - \bu(\bx)}{|\by - \bx|} \cdot \frac{\by - \bx}{|\by - \bx|} = e(\bu)(\bx) \frac{\by - \bx}{|\by - \bx|} \cdot \frac{\by - \bx}{|\by - \bx|} + \mathcal{O}(|\by - \bx|) \quad \text{as } \by \to \bx.
\end{equation*}
This directional projection appears in applications where it serves as a nonlocal measure of linearized strain and forms a core building block of the elastic energies studied in linearized peridynamics \cite{Silling2000,Silling2007,Silling2010}. Our interest in norms based on these projected difference quotients is therefore driven by the need to understand the coercivity and fundamental analytical properties of the associated energy spaces. 

{
To that end, we introduce the analogous Gagliardo seminorm given by
\begin{equation}\label{eq:NonlocalSymSeminormDef}
    [\bu]_{\mathcal{X}^{s,p}(\Omega)}^p := \int_{\Omega} \int_{\Omega} \frac{\left| (\bu(\by) - \bu(\bx)) \cdot \frac{\by - \bx}{|\by - \bx|} \right|^p}{|\by - \bx|^{n+sp}} \, \mathrm{d}\by \, \mathrm{d}\bx,
\end{equation}
and the associated Banach space is
\begin{equation}
    \mathcal{X}^{s,p}(\Omega) := \left\{ \bu \in L^p(\Omega;\mathbb{R}^n) : [\bu]_{\mathcal{X}^{s,p}(\Omega)} < \infty \right\},
\end{equation}
equipped with the norm $\|\bu\|_{\mathcal{X}^{s,p}(\Omega)}^p := \|\bu\|_{L^p(\Omega)}^p + [\bu]_{\mathcal{X}^{s,p}(\Omega)}^p$. The rescaled seminorm $(1-s)^{\frac{1}{p}}[\bu]_{\mathcal{X}^{s,p}}$ is shown to converge as $s\to 1$ to a seminorm equivalent to $\|e(\bu)\|_{L^{p}(\Omega)}$, \cite{Mengesha2012}. Moreover, on any domain, the zero-set of $[\bu]_{\mathcal{X}^{s,p}(\Omega)}$ consists precisely of the space of infinitesimal rigid motions $\mathcal{R}$ \cite{Mengesha2012, Temam-Miran}.
\subsection{Statement of main results}
Our first main result is a fractional analog of Korn's first unconstrained inequality on arbitrary bounded Lipschitz domains. Furthermore, our estimate is stable in the limit $s \to 1$, which in turn implies the local inequality \eqref{eq:Intro-Korn1stCase2}. 
While we previously obtained this inequality in \cite{MS2022} for domains with sufficiently small Lipschitz constant, both the general domain case and the stability of the estimate as $s \to 1$ remained unanswered in that work. In the present work, we completely eliminate this geometric constraint and  establish stability in $s$.

\begin{theorem}\label{thm:mainKorn1st}
    Let $n \geq 2$, and let $\Omega \subset \mathbb{R}^n$ be a bounded Lipschitz domain. Let $1<p<\infty$ and $s \in (0,1)$.
    Then there exists a constant $C = C(n,p,\Omega) > 0$ such that for all vector fields $\bu \in W^{s,p}(\Omega;\bbR^n)$,
    \begin{equation}\label{eq:mainKorn1stCase2}
        \inf_{\bspsi\in \mathcal{R}} \left( (1-s)^{-1} \vnorm{\bu- \bspsi}_{L^p(\Omega)}^p + [\bu - \bspsi]_{W^{s,p}(\Omega)}^p \right) \leq C[\bu]_{\mathcal{X}^{s,p}(\Omega)}^p.
    \end{equation}
    Moreover, there exists a bounded linear projection $\Pi : L^p(\Omega) \to \cR$ with $\vnorm{\Pi[\bu]}_{\mathrm{aff}} \leq C \vnorm{\bu}_{L^p(\Omega)}$ such that \begin{equation}\label{eq:mainKorn1stCase2:Proj}
         (1-s)^{-1}\vnorm{\bu- \Pi[\bu]}_{L^p(\Omega)}^p + [\bu - \Pi[\bu]]_{W^{s,p}(\Omega)}^p  \leq C [\bu]_{\mathcal{X}^{s,p}(\Omega)}^p.
    \end{equation}
\end{theorem}
The norm of an affine map $\bspsi(\bx) = \bA \bx + \bfb$ is defined as the sum of Frobenius norm and Euclidean norm: $\vnorm{\bspsi}_{\mathrm{aff}} := |\bA| + |\bfb|$.
Notice that after multiplying through by $(1-s)$, 
 the inequality \eqref{eq:mainKorn1stCase2:Proj} is stable as $s\to 1$ and recovers the classical Korn's inequality \eqref{eq:Intro-Korn1stCase2:Proj} for functions $\bu$ in the local Sobolev spaces $W^{1,p}$.

Our second main result is the same fractional Korn's inequality in the borderline cases $p = 1$ and $p = \infty$. To the best of our knowledge, these inequalities are the first of their kind. Moreover, they have no classical local analog , as Korn's inequalities are well-known to fail for $p=1$ and $p=\infty$.
\begin{theorem}\label{thm:mainKorn1st:Endpoints}
    Let $n \geq 2$ and let $\Omega \subset \bbR^n$ be a bounded Lipschitz domain. Let $p=1$ or $p=\infty$
    and $s \in (0,1)$. Then there exists a constant $C = C(n,\Omega) > 0$ such that for all $\bu \in W^{s,p}(\Omega;\bbR^n)$
    \begin{equation}\label{eq:mainKorn1stCase2:Proj:Endpoints}
         (1-s)^{-1} \vnorm{\bu- \Pi[\bu]}_{L^p(\Omega)}^p + [\bu - \Pi[\bu]]_{W^{s,p}(\Omega)}^p \leq \frac{C}{1-s} [\bu]_{\mathcal{X}^{s,p}(\Omega)}^p, \quad \forall \bu \in W^{s,p}(\Omega;\bbR^n),
    \end{equation}
    where $\Pi$ is a bounded linear projection onto $\cR$ with $\vnorm{\Pi[\bu]}_{\mathrm{aff}} \leq C \vnorm{\bu}_{L^p(\Omega)}$.
\end{theorem}
Since $\Omega$ is bounded we identify $W^{s,\infty}(\Omega;\bbR^n)$ as $C^s(\Omega;\bbR^n)$, the space of H\"older-continuous functions with exponent $s$. $\cX^{s,\infty}$ is defined analogously; the details are in Section \ref{subsec:infty}.
A fundamental difference between \eqref{eq:mainKorn1stCase2:Proj} and \eqref{eq:mainKorn1stCase2:Proj:Endpoints} is the constant's dependence on $s$.
For smooth functions, after multiplying \eqref{eq:mainKorn1stCase2:Proj:Endpoints} by $(1-s)$, the left-hand side remains bounded by the $W^{1,p}$ norm while the right-hand side is unbounded 
as $s \to 1$ (see \cite{Mengesha2012}). Thus \eqref{eq:mainKorn1stCase2:Proj:Endpoints} is consistent with the failure of the classical Korn's inequality for $W^{1,1}$ and $W^{1,\infty}$.

From Theorem \ref{thm:mainKorn1st} and Theorem \ref{thm:mainKorn1st:Endpoints}, we deduce the following fractional versions of Korn's second inequality and Korn's first constrained inequality, and establish the function space equivalence $\mathcal{X}^{s,p}(\Omega) = W^{s,p}(\Omega;\mathbb{R}^n)$.

\begin{theorem}\label{thm:mainKorn2nd}
    With $n$, $\Omega$, and $s$ as in Theorem \ref{thm:mainKorn1st}, let $1\leq p\leq \infty.$
    Then there exists a constant $C = C(n,p,\Omega) > 0$ such that for all $\bu \in \mathcal{X}^{s,p}(\Omega)$,
    \begin{equation}\label{eq:mainKorn2nd}
        [\bu]_{W^{s,p}(\Omega)}^p \leq C \left( \alpha_{s}(p) [\bu]_{\mathcal{X}^{s,p}(\Omega)}^p + (1-s)^{-1} \|\bu\|_{L^p(\Omega)}^p \right),
    \end{equation}
    where \[\alpha_s(p) = \begin{cases}
  1, & \text{if } 1 < p < \infty, \\[1ex]
  (1-s)^{-1}, & \text{if } p=1 \text{ or } p=\infty. 
\end{cases}\]
    Consequently, $\mathcal{X}^{s,p}(\Omega) = W^{s,p}(\Omega;\mathbb{R}^n)$ with equivalent norms.
\end{theorem}

\begin{corollary}\label{cor:mainKorn1st:Constrained}
    With $n$, $\Omega$, and $s$ as in Theorem \ref{thm:mainKorn1st}, let $p \in [1,\infty)$. 
    Suppose that  $\mathring{\mathcal{X}}^{s,p}(\Omega) \subset \mathcal{X}^{s,p}(\Omega)$ is any closed subspace equipped with the $\mathcal{X}^{s,p}(\Omega)$-norm such that $\mathring{\mathcal{X}}^{s,p}(\Omega) \cap \mathcal{R} = \{ \mathbf{0} \}$. Then there exists a constant $C = C (\mathring{\mathcal{X}}^{s,p}(\Omega)) > 0$ such that \begin{equation}\label{eq:mainKorn1stCase1}
            \|\bu\|_{W^{s,p}(\Omega)}^p \leq C [\bu]_{\mathcal{X}^{s,p}(\Omega)}^p \quad \text{for all } \bu \in \mathring{\mathcal{X}}^{s,p}(\Omega).
        \end{equation}
\end{corollary}

Efforts to establish estimates of the form \eqref{eq:mainKorn1stCase2}, \eqref{eq:mainKorn2nd}, and \eqref{eq:mainKorn1stCase1} have been incremental over recent years. Prior to \cite{MS2022}, such estimates were proved for restricted classes of functions such as $W^{s,p}_0(\Omega;\mathbb{R}^n)$ under specific parameter constraints on $s$, $p$, and the domain $\Omega$ \cite{M-half,Rut.2022,MScott2018Korn}. These earlier proofs relied fundamentally on the fractional Korn inequality on the entire space $\mathbb{R}^n$ established by the second and fourth authors in \cite{MScott2018Korn}. In a recent independent study, Acosta et al. \cite{Duran2026} also investigated fractional Korn inequalities on uniform domains.
For Lipschitz domains, the results in this work strengthen previous results in the setting of $p \in (1,\infty)$, and are completely new for $p = 1,\infty$.

\subsection{Constrained fractional Korn inequalities}\label{subsec:subspaces}
Using Corollary \ref{cor:mainKorn1st:Constrained}, we systematically examine several classes of constrained subspaces $\mathring{\mathcal{X}}^{s,p}(\Omega)$ that arise in nonlocal variational problems and physical applications. A prime motivational application is peridynamics \cite{Silling2000,Silling2007,Silling2010}, a nonlocal continuum mechanics theory whose linearized constitutive laws are typically characterized via projected difference quotients. In variational formulations of peridynamics, the nonlocal analog  of classical Dirichlet boundary conditions is customarily imposed through \textit{volume constraints} on $\mathbf{u}$, while for Neumann-type nonlocal problems, rigid motions are eliminated by restricting the fields to the kernel of an appropriate linear projection operator. The next two propositions establish the equivalence of norms on these classes of vector fields.

\paragraph{Nonlocal volume (peridynamic collar) constraints.}
In nonlocal models, minimizers can lack sufficient regularity in order to properly prescribe boundary constraints. Instead, a common practice is to state boundary conditions for $\bu$ on a ``collar'' of positive Lebesgue measure $\widetilde{\Omega} \subset \Omega$.

\begin{proposition}\label{prop:PeridynamicCollar}
    Under the hypotheses of Corollary \ref{cor:mainKorn1st:Constrained}, let $\widetilde{\Omega} \subset \Omega$ be a measurable set with $|\widetilde{\Omega}| > 0$. Define
    \begin{equation}
        \widetilde{\mathcal{X}}^{s,p}_{\widetilde{\Omega}}(\Omega) := \left\{ \bu \in \mathcal{X}^{s,p}(\Omega) : \bu = \mathbf{0} \text{ a.e. on } \widetilde{\Omega} \right\}.
    \end{equation}
    Then there exists $C = C(n,s,p,\Omega,\widetilde{\Omega}) > 0$ such that \eqref{eq:mainKorn1stCase1} holds for all $\bu \in \widetilde{\mathcal{X}}^{s,p}_{\widetilde{\Omega}}(\Omega)$.
\end{proposition}

\paragraph{Continuous linear functional constraints.}
For Neumann-type nonlocal problems where rigid motions cannot be eliminated by vanishing trace or volume  conditions on the vector field $\bu$, we have the following result.  

\begin{proposition}\label{prop:FunctionalConstraint}
    Under the hypotheses of Corollary \ref{cor:mainKorn1st:Constrained}, let $P : \mathcal{X}^{s,p}(\Omega) \to \mathcal{R}$ be any continuous linear operator such that $P|_{\mathcal{R}} = \mathrm{id}_{\mathcal{R}}$. Define
    \begin{equation}
        \dot{\mathcal{X}}^{s,p}(\Omega) := \left\{ \bu \in \mathcal{X}^{s,p}(\Omega) : P[\bu] = \mathbf{0} \right\}.
    \end{equation}
    Then there exists $C = C(n,s,p,\Omega,P) > 0$ such that \eqref{eq:mainKorn1stCase1} holds for all $\bu \in \dot{\mathcal{X}}^{s,p}(\Omega)$.
\end{proposition}

\paragraph{Trace-zero vector fields.}
Let $\Gamma \subset \partial\Omega$ be a measurable boundary subset with positive $(n-1)$-dimensional Hausdorff measure $\mathcal{H}^{n-1}(\Gamma) > 0$. We define $\mathcal{X}^{s,p}_{0,\Gamma}(\Omega)$ as the closure in $\mathcal{X}^{s,p}(\Omega)$ of the space of smooth vector fields vanishing in a neighborhood of $\overline{\Gamma}$. 
When $\Gamma = \partial\Omega$, we denote $\mathcal{X}^{s,p}_{0,\Gamma}(\Omega)$ simply by $\mathcal{X}^{s,p}_0(\Omega)$. The validity of Korn's first inequality \eqref{eq:mainKorn1stCase1} on $\mathcal{X}^{s,p}_0(\Omega)$ was analyzed by the first and third authors in \cite{HM2023} for $C^1$ domains when $sp > 1$. They also demonstrated via a Kirszbraun-type construction \cite{Kirzbraun1934} that \eqref{eq:mainKorn1stCase1} fails for \textit{any} open bounded $C^1$ domain whenever $sp < 1$.

Through Theorem \ref{thm:mainKorn1st} and Theorem \ref{thm:mainKorn2nd}, we recover and extend these trace results to arbitrary bounded Lipschitz domains, including the critical boundary case $sp = 1$. By norm equivalence, $\mathcal{X}^{s,p}_0(\Omega) = W^{s,p}_0(\Omega;\mathbb{R}^n)$. When $sp > 1$, vector fields in $W^{s,p}_0(\Omega;\mathbb{R}^n)$ possess vanishing trace on $\partial\Omega$, which forces $\mathcal{R} \cap \mathcal{X}^{s,p}_0(\Omega) = \{\mathbf{0}\}$ and yields \eqref{eq:mainKorn1stCase1}. Conversely, when $sp \le 1$, 
$W^{s,p}_0(\Omega;\mathbb{R}^n) = W^{s,p}(\Omega;\mathbb{R}^n)$ (see \cite{leoni2023first}), which implies $\mathcal{R} \subset \mathcal{X}^{s,p}_0(\Omega)$ and causes \eqref{eq:mainKorn1stCase1} to fail. We state this dichotomy precisely:

\begin{proposition}\label{prop:ZeroTraceSubspace}
    Under the hypotheses of Corollary \ref{cor:mainKorn1st:Constrained}, let $\Gamma \subset \partial\Omega$ satisfy $\mathcal{H}^{n-1}(\Gamma) > 0$.
    \begin{enumerate}
        \item[\textbf{(a)}] If $sp > 1$, there exists $C = C(n,s,p,\Omega,\Gamma) > 0$ such that
        \begin{equation}\label{eq:ZeroTraceEst}
            \|\bu\|_{W^{s,p}(\Omega)}^p \leq C [\bu]_{\mathcal{X}^{s,p}(\Omega)}^p \quad \text{for all } \bu \in \mathcal{X}^{s,p}_{0,\Gamma}(\Omega).
        \end{equation}
        \item[\textbf{(b)}] If $sp \le 1$, inequality \eqref{eq:mainKorn1stCase1} fails on $\mathcal{X}^{s,p}_{0,\Gamma}(\Omega)$. 
    \end{enumerate}
\end{proposition}

\paragraph{The critical case \texorpdfstring{$sp=1$}{sp=1} and Lions-Magenes spaces.}
When $sp \neq 1$, vector fields belonging to $W^{s,p}_0(\Omega;\mathbb{R}^n)$ extended by zero outside $\Omega$ belong to $W^{s,p}(\mathbb{R}^n;\mathbb{R}^n)$. However, in the critical case $sp=1$, extension by zero of functions in $W^{s,1/s}_0(\Omega;\mathbb{R}^n)$ may fail to lie in $W^{s,1/s}(\mathbb{R}^n;\mathbb{R}^n)$. On Lipschitz domains, the subspace of vector fields whose zero extensions belong to $W^{s,1/s}(\mathbb{R}^n;\mathbb{R}^n)$ forms a proper subspace denoted by $W^{s,1/s}_{00}(\Omega;\mathbb{R}^n)$, known as the Lions-Magenes space \cite{LionsPeetre1964,LionsMagenes1968}. Equipped with the norm
\begin{equation}
    \|\bu\|_{W^{s,1/s}_{00}(\Omega)}^{1/s} := \int_{\Omega} \int_{\mathbb{R}^n \setminus \Omega} \frac{|\bu(\bx)|^{1/s}}{|\bx - \by|^{n+1}} \, \mathrm{d}\by \, \mathrm{d}\bx + [\bu]_{W^{s,1/s}(\Omega)}^{1/s}, 
\end{equation}
$W^{s,1/s}_{00}(\Omega;\mathbb{R}^n)$ is a Banach space. The $W^{s,1/s}_{00}(\Omega)$-norm is strictly stronger than the standard $W^{s,1/s}(\Omega)$-norm, see \cite[Section 2.7]{leoni2023first} for a recent and self-contained reference. 
As a consequence, an estimate of the form \eqref{eq:mainKorn1stCase1} on $W^{s,1/s}(\Omega)$ would not capture the full $W^{s,1/s}_{00}(\Omega)$-norm. Nevertheless, we establish a specialized fractional Korn's inequality controlling the complete $W^{s,1/s}_{00}(\Omega)$ norm:

\begin{proposition}\label{prop:LionsMagenesKorn}
    Under the hypotheses of Corollary \ref{cor:mainKorn1st:Constrained}, let $sp=1$. Define $\mathcal{X}^{s,1/s}_{00}(\Omega) := \{ \bu \in \mathcal{X}^{s,1/s}(\Omega) : \widetilde{\bu} \in \mathcal{X}^{s,1/s}(\mathbb{R}^n) \}$, where $\widetilde{\bu}$ denotes the extension of $\bu$ to $\mathbb{R}^n \setminus \Omega$ by $\mathbf{0}$. Then there exists $C = C(n,s,\Omega) > 0$ such that
    \begin{equation}
        \|\bu\|_{W^{s,1/s}_{00}(\Omega)} \leq C \|\bu\|_{\mathcal{X}^{s,1/s}_{00}(\Omega)} \quad \text{for all } \bu \in \mathcal{X}^{s,1/s}_{00}(\Omega),
    \end{equation}
    where $\displaystyle \|\bu\|_{\mathcal{X}^{s,1/s}_{00}(\Omega)}^{1/s} := \int_{\Omega} \int_{\mathbb{R}^n \setminus \Omega} \frac{\left| \bu(\bx) \cdot \frac{\bx-\by}{|\bx-\by|} \right|^{1/s}}{|\bx-\by|^{n+1}} \, \mathrm{d}\by \, \mathrm{d}\bx + [\bu]_{\mathcal{X}^{s,1/s}(\Omega)}^{1/s}$.
\end{proposition}

\subsection{Quantitative improvements of fractional Korn's inequalities}\label{subsec:narrative-strategy}
We describe the roadmap followed in this work to establish the fractional Korn's inequalities. First, the fractional Korn's second inequality and the resulting space equivalence (Theorem~\ref{thm:mainKorn2nd}) are deduced from \eqref{eq:mainKorn1stCase2:Proj} via the boundedness of the explicit projection operator from $L^p$ onto $\cR$.
The fractional Korn's first constrained inequalities (Corollary \ref{cor:mainKorn1st:Constrained} and Propositions \ref{prop:PeridynamicCollar}--\ref{prop:LionsMagenesKorn}) are then derived from Theorem \ref{thm:mainKorn2nd} via compactness arguments. The primary analytical challenge is therefore proving the fractional Korn's first unconstrained inequalities \eqref{eq:mainKorn1stCase2:Proj} for $p \in (1,\infty)$ and \eqref{eq:mainKorn1stCase2:Proj:Endpoints} for $p \in \{1,\infty\}$.

We prove these inequalities in two major steps. First, we prove \eqref{eq:mainKorn1stCase2:Proj} and \eqref{eq:mainKorn1stCase2:Proj:Endpoints} in the case that $\Omega$ is a Euclidean ball. 
Qualitatively, the related estimate \eqref{eq:mainKorn1stCase2} has already been established in previous works by (a subset of) the authors \cite{MS2022}. In those works, the result relies on constructing a bounded linear extension operator from $B$ to all of $\bbR^n$ that preserves the $\cX^{s,p}$-seminorm, and then appealing to a fractional Korn-type inequality on $\bbR^n$ proved by (a subset of) the authors in \cite{MScott2018Korn}. The universal constants in this entire space Korn inequality blow up as $s \to 1$, and so any inequality whose proof makes use of the whole-space estimate as an ingredient inherits this degeneracy. For that reason, none of the results in this work make use of any extension arguments. Our methods depart from existing techniques in the literature. The inequality is ``built from scratch'' via Fredholm theory on $L^p$-bounded operators which characterize projected differences. These operators are then localized on dyadic decompositions around the diagonal $\bx=\by$ of the Gagliardo seminorms; the resulting scaling gives the explicit dependence on the differentiability $s$.

To further illustrate the dependence of \eqref{eq:mainKorn1stCase2:Proj} and \eqref{eq:mainKorn1stCase2:Proj:Endpoints} on $s$, the explicitly-constructed projection operator $\Pi : L^1(\Omega) \to \cR$ defined in Lemma \ref{lem:ProjectionBall} is shown to be continuous. Therefore, $\bu - \Pi[\bu]$ belongs to $\cX^{s,p}$ whenever $\bu \in \cX^{s,p}$, and the inequalities hold for the same projection $\Pi$ for the entire range of $s$ and $p$. This is necessary for showing consistency in the nonlocal-to-local limit as well as passing to the limit in the compactness arguments of the constrained Korn inequalities.

The second major step is to prove
use \eqref{eq:mainKorn1stCase2:Proj} and \eqref{eq:mainKorn1stCase2:Proj:Endpoints} in the case of bounded Lipschitz domains. For $p = \infty$ the Korn inequality is proved directly, while for $p < \infty$ the inequality is first proved on the bi-Lipschitz image of a ball, making use of a locally-finite covering argument. It bears similarity to a Whitney decomposition, but uses Euclidean balls rather than disjoint cubes, so that the same Korn inequalities established in the first major step can be used. The result (Theorem \ref{thm:KornBiLipBall}) additionally uses a fractional norm characterization for matrices in order to control the difference in rigid motions across overlapping elements of the cover.

We highlight key differences between our approach and the recent preprint \cite{Duran2026}, which proves a fractional Korn's inequality of the form \eqref{eq:mainKorn1stCase2} for $p \in (1, \infty)$ on the more general class of uniform domains.
A primary distinction is the choice of projection onto $\mathcal{R}$. In \cite{Duran2026}, the projection is bounded only on the smaller space $W^{s,p}$, making it unclear how to directly deduce the second fractional Korn's inequality from \eqref{eq:mainKorn1stCase2}.  Moreover, their result applies solely to vector fields  
in $W^{s,p}(\Omega)$,  
and not to those that \textit{a priori} reside only in $\cX^{s,p}(\Omega)$. The use of the projection $\Pi$ removes both these difficulties in our arguments. 
The global covering arguments in this work and \cite{Duran2026} also differ in their implementation. Rather than smoothing Whitney cubes and controlling energies on the chains via a global discrete Poincar{\'e} inequality as in \cite{Duran2026}, we employ a covering using Euclidean balls from \cite{Kohn1982} with established chaining properties and rely on local matrix norm comparisons on intersecting elements. 

The paper is organized as follows:
In Section \ref{sec:ball}, we construct the explicit projection operator $\Pi_B$ and prove the strengthened fractional Korn's first unconstrained inequality on Euclidean balls (Theorem \ref{thm:KornBallStrong}) for $p \in (1,\infty)$. The cases $p = 1$ and $p = \infty$ are proved in Sections \ref{subsec:1} and \ref{subsec:infty}, respectively.
In Section \ref{sec:lipschitz}, we establish the Korn inequality on bi-Lipschitz images of balls (Theorem \ref{thm:KornBiLipBall}) and complete the proofs of Theorems \ref{thm:mainKorn1st}-\ref{thm:mainKorn2nd}
for bounded Lipschitz domains.
Finally, in Section \ref{sec:subspaces}, we present the proofs of the subspace Korn inequalities and characterizations (Corollary \ref{cor:mainKorn1st:Constrained} and Propositions \ref{prop:ZeroTraceSubspace}--\ref{prop:LionsMagenesKorn}).

\paragraph{Acknowledgments} 
J.M.S. acknowledges support from NSF DMS-2606364. T.M.'s research is supported by NSF DMS-2509059. The authors would like to thank Davit Harutyunyan and Hayk Mikayelyan for their contributions during the initial stages of portions of the work presented in Section \ref{sec:lipschitz}.

\paragraph{AI Disclosure} The results of Sections \ref{subsec:p}, \ref{subsec:1}, and \ref{subsec:infty} were found by modifying outputs from a series of prompts placed into ChatGPT-5.6 Sol; all other results of the manuscript were found without the use of LLMs. The entire manuscript was written by the authors. The authors take full responsibility for the correctness of all the results and proofs of this paper.

\section{Fractional Korn's first inequality on Euclidean balls}\label{sec:ball}

In \cite{MS2022-HMMS}, it was shown that for any ball $B(\bx_0, R) \subset \mathbb{R}^n$, a fractional analog  of Korn's first 
unconstrained inequality holds: 
specifically, for $p \in (1,\infty)$, there exists $C = C(n, p,s) > 0$ such that  \[\inf_{\bspsi\in \mathcal{R}} \|\bu - \bspsi\|_{W^{s,p}(B(\bx_0, R))} \leq C [\bu]_{\mathcal{X}^{s,p}(B(\bx_0, R))}\]
for all $\bu \in W^{s,p}(B(\bx_0, R))$.
Here we strengthen this inequality in three ways. First, we show that the dependence of the constant $C$ on $s$ can be chosen so that the fractional inequality recovers the classical Korn's inequality for $\bu \in W^{1,p}(B(\bx_0, R);\bbR^n)$ as $s\to 1$. 
Second, we improve the fractional Korn's first inequality by replacing the infimum over $\mathcal{R}$ with an explicit constructed projection operator $\Pi_{B(\bx_0, R)}$, while at the same time still ensuring the constant is stable as $s\to 1$. Third, we show that the inequality additionally holds true for $p = 1$ and $p = \infty$, while capturing the blow up of the constant $C$ as $s \to 1$, thus remaining consistent with the lack of classical Korn's inequalities in those cases.

Throughout this section, we set  $B = B(0,1) \subset \mathbb{R}^n$ be the unit ball. 

\subsection{Construction of the projection operator \texorpdfstring{$\Pi_{B}$}{}}
We recall the space of infinitesimal rigid motions 
\begin{equation*}
    \mathcal{R} := \{ \bspsi(\bx) = \mathbf{A}\bx + \mathbf{b} : \bA \in \bbR^{n \times n}, \, \mathbf{b} \in \mathbb{R}^n, \, \mathbf{A} + \bA^T = \bzero \}.
\end{equation*}
This is a finite dimensional space with norm $\|\mathbf{A}\bx+\mathbf{b}\|_{\mathrm{aff}} := |\mathbf{A}| + |\mathbf{b}|$, where $|\mathbf{A}|$ represents any matrix norm. Moreover, since all norms are equivalent in this space, for any affine map $\bspsi = \bA \bx + \bfb$, elementary estimates and scaling arguments show that for all $n \in \bbN$ and for all $1\leq p\leq \infty$
    \begin{equation}
        \frac{1}{C(n,p)} \vnorm{\bspsi}_{L^p(B}  \leq \vnorm{\bspsi}_{\mathrm{aff}} \leq C(n,p) \vnorm{\bspsi}_{L^p(B} 
    \end{equation}
    and for all $R > 0$
\begin{equation}\label{eq:AffineMapEquivalence:Rescaled}
        \frac{1}{C(n,p) R^{n/p} \max\{1,R\} } \vnorm{\bspsi}_{L^p(B(0,R))} \leq  \vnorm{\bspsi}_{\mathrm{aff}} \leq \frac{C(n,p)}{ R^{n/p} \min\{1,R\} } \vnorm{\bspsi}_{L^p(B(0,R))}.
    \end{equation}
    By simple translation, we also have that for $\bx_0 \in \bbR^n$
\begin{equation}\label{eq:AffineMapEquivalence:Rescaled:MatrixPart}
        \frac{1}{C(n,p) R^{n/p+1}} \vnorm{\bspsi}_{L^p(B(\bx_0,R))} \leq  |\bA| \leq \frac{C(n,p)}{R^{n/p+1} } \vnorm{\bspsi}_{L^p(B(\bx_0,R))},
    \end{equation}
    where we take $n/p = 0$ in the case $p = \infty$.

We are now ready to define the projection operator $\Pi_{B} : L^1(B;\mathbb{R}^n) \to \mathcal{R}$. 
First, define the average of $\bu$ on $B$ as 
    \begin{equation}
        \mathbf{\varrho}(\bu) := \fint_{B} \bu(\bx) \, \mathrm{d}\bx = \frac{1}{|B|}\int_{B} \bu(\bx) \, \mathrm{d}\bx,
    \end{equation}
    and let $\mathbf{\Sigma}(\bu) = (\sigma^{ij}(\bu))_{1 \leq i,j \leq n}$ be the matrix with entries
    \begin{equation}
        \sigma^{ij}(\bu) := \frac{1}{\kappa_n} \int_{B(0,1) \cap \{x_j > 0\}} \big( u^i(\bx) - \varrho^i(\bu) \big) \, \mathrm{d}\bx,
    \end{equation}
    where $\kappa_n$ is a normalization constant  \footnote{In fact, $\kappa:= \frac{\omega_{n-2}}{(n-1)(n+1)}$, where $\omega_d$ being the surface area of $\mathbb{S}^{d-1} \subset \mathbb{R}^n$} so that 
   $\displaystyle\int_{B \cap \{x_j > 0\}} x_j \, \mathrm{d}\bx = \kappa_n$. Set
    \begin{equation}\label{eq:CanonProjFormula}
        \Pi_{B}[\bu](\bx) := \frac{\mathbf{\Sigma}(\bu) - \mathbf{\Sigma}(\bu)^\top}{2} \bx + \mathbf{\varrho}(\bu),\qquad \bx\in \mathbb{R}^n.
    \end{equation}
For a general ball $B(\bx_0, R) \subset \mathbb{R}^n$, the operator $\Pi_{B(\bx_0,R)} : L^1(B(\bx_0,R);\mathbb{R}^n) \to \mathcal{R}$ defined by
    \begin{equation}
        \Pi_{B(\bx_0,R)}[\bu](\bx) := \Pi_{B}[\mathbf{v}]\left( \frac{\bx - \bx_0}{R} \right), \quad \text{where } \mathbf{v}(\mathbf{z}) = \bu(\bx_0 + R\mathbf{z}). 
    \end{equation}
    \begin{lemma}\label{lem:ProjectionBall}
    The operator  $\Pi_{B}$ is a continuous linear operator from $L^1(B;\mathbb{R}^n)$ onto $\mathcal{R}$ satisfying
    \begin{equation}
        \|\Pi_{B}[\bu]\|_{\mathrm{aff}} \leq C(n) \|\bu\|_{L^1(B)}.
    \end{equation}
  Furthermore, for any affine map $\bspsi(\bx) = \mathbf{A}\bx + \mathbf{b}$,
    \begin{equation}
        \Pi_{B}[\bspsi](\bx) = \frac{\mathbf{A} - \mathbf{A}^\top}{2} \bx + \mathbf{b}.
    \end{equation}
    In particular, $\Pi_{B}[\bspsi] = \bspsi$ for all $\bspsi\in \mathcal{R}$.
    
    For a general ball $B(\bx_0, R)$, similar results hold: $\Pi_{B(\bx_0,R)}[\bspsi] = \bspsi$ for all $\bspsi\in \mathcal{R}$ and for all $1\leq p\leq \infty$, 
    \begin{equation}\label{eq:ProjContLp}
        \|\Pi_{B(\bx_0,R)}[\bu]\|_{L^p(B(\bx_0,R))} \leq  C(n,p) \|\bu\|_{L^p(B(\bx_0,R))}. 
    \end{equation}
   
\end{lemma}

\begin{proof}
    Direct computation in polar coordinates verifies $\displaystyle\int_{B(0,1) \cap \{x_j > 0\}} x_k \, \mathrm{d}\bx = \kappa_n \delta_{jk}$. For an affine function $\bspsi(\bx) = \mathbf{A}\bx + \mathbf{b}$, we have $\mathbf{\varrho}(\bspsi) = \mathbf{b}$ and $\displaystyle\sigma^{ij}(\bspsi) = \frac{1}{\kappa_n} \int_{B(0,1) \cap \{x_j > 0\}} [\mathbf{A}\bx]^i \, \mathrm{d}\bx = A_{ij}$. Thus $\mathbf{\Sigma}(\bspsi) = \mathbf{A}$, so \eqref{eq:CanonProjFormula} yields $\displaystyle\Pi_{B(0,1)}[\bspsi] = \frac{\mathbf{A}-\mathbf{A}^\top}{2}\bx + \mathbf{b}$. If $\mathbf{A} \in \mathrm{Skew}(n)$, then $\displaystyle\frac{\mathbf{A}-\mathbf{A}^\top}{2} = \mathbf{A}$, so $\displaystyle\Pi_{B(0,1)}[\bspsi] = \bspsi$. Continuity on $L^1$ follows immediately from the triangle and Hölder inequalities.
\end{proof}

We will also use the following relationship between the $W^{s,p}$ seminorm and the affine norm.

\begin{lemma}
    For any matrix $\bA \in \bbR^{n \times n}$, $s \in (0,1)$, and $p \in [1,\infty)$,
    \begin{equation}\label{eq:NormCompEquiv}
        C_1(n,p) |\bA| \leq \left( (1-s) \int_{B} \fint_{B} \frac{|\bA \bx - \bA \by|^p}{|\bx-\by|^{n+sp} } \, \rmd \by \, \rmd \bx \right)^{1/p} \leq C_2(n,p) |\bA|.
    \end{equation}
\end{lemma}

\begin{proof}
    To prove the upper bound in \eqref{eq:NormCompEquiv}, using $|\mathbf{A}\bx - \mathbf{A}\by| \leq |\mathbf{A}| |\bx-\by|$, we integrate:
    \begin{equation}
        \fint_{B} \int_{B} \frac{|\mathbf{A}\bx - \mathbf{A}\by|^p}{|\bx-\by|^{n+sp}} \, \mathrm{d}\by \, \mathrm{d}\bx \leq |\mathbf{A}|^p \fint_{B)} \int_{B(\bx,2)} \frac{1}{|\bx-\by|^{n-(1-s)p}} \, \mathrm{d}\by \, \mathrm{d}\bx \leq \frac{C(n,p)}{1-s} |\mathbf{A}|^p.
    \end{equation}
    To prove the lower bound in \eqref{eq:NormCompEquiv}, we restrict the domain of integration:
    \begin{equation}\label{eq:NormCompEquiv:Pf1}
        \begin{aligned}
            \fint_{B} \int_{B} \frac{|\mathbf{A}\bx - \mathbf{A}\by|^p}{|\bx-\by|^{n+sp}} \, \mathrm{d}\by \, \mathrm{d}\bx
            &\geq \fint_{B(0,1/2)} \int_{B(\bx,1/2)} \frac{|\mathbf{A}(\bx-\by)|^p}{|\bx-\by|^{n+sp}} \, \mathrm{d}\by \, \mathrm{d}\bx \\
            &= \frac{C(n)}{2^{(1-s)p}(1-s)p} \int_{\mathbb{S}^{n-1}} |\mathbf{A}\mathbf{\omega}|^p \, \mathrm{d}\sigma(\mathbf{\omega}).
        \end{aligned}
    \end{equation}
    Now, $\displaystyle\int_{\bbS^{n-1}} \omega \otimes \omega \, \rmd \sigma(\omega) = C(n) \bI$, so by H\"older's inequality
    \begin{equation*}
        |\bA|^p = C(n,p) \left| \int_{\bbS^{n-1}} (\bA \omega) \otimes \omega \, \rmd \sigma(\omega) \right|^p \leq C(n,p) \int_{\bbS^{n-1}} |\bA \omega|^p \, \rmd \sigma(\omega).
    \end{equation*}
    This inequality combined with \eqref{eq:NormCompEquiv:Pf1} completes the proof of the lower bound in \eqref{eq:NormCompEquiv}.
\end{proof}

The same result for the case $p = \infty$ follows easily from the equivalence of the matrix norms $|\bA|$ and $\sup_{|\bx| = 1} |\bA\bx|$:

\begin{lemma}
    For any matrix $\bA \in \bbR^{n \times n}$ and $s \in (0,1]$
    \begin{equation}\label{eq:NormCompEquiv:Infty}
        n^{-1/2} |\bA| \leq \sup_{\substack{\bx,\by \in B(0,1) \\ \bx \neq \by}} \frac{|\bA(\bx-\by)|}{|\bx-\by|^s} \leq 2 |\bA|.
    \end{equation}
\end{lemma}

\subsection{Fractional Korn's first inequality on Balls: the case \texorpdfstring{$1 < p < \infty$}{1<p<infinity}}\label{subsec:p}

Throughout this subsection, we assume $1\leq p < \infty$ unless otherwise specified.
We first have several lemmas concerning the coercivity of functionals:

For $r>0$,  a measurable set $\Omega \subset \bbR^n$, $p \in [1,\infty)$ and vector field $\bu : \Omega \to \bbR^n$, define the functionals  
$$
\begin{aligned}
\cD_r(\bu;\Omega)&:=\int_\Omega \int_\Omega
\mathbf 1_{\{r/2<|\bx-\by|<r\}}|\bu(\bx)-\bu(\by)|^p\, \rmd \by\,\rmd \bx,\\
\cE_r(\bu;\Omega)&:=\int_\Omega \int_\Omega
\mathbf 1_{\{r/8<|\bx-\by|<4r\}}
\left|(\bu(\bx)-\bu(\by))\cdot\frac{\bx-\by}{|\bx-\by|}\right|^p\, \rmd \by\,\rmd \bx .
\end{aligned}
$$ 

The following lemma characterizes $\cR$ as a zero sent of $\cE_r(\bu;B)$. Characterization of this type via averaging of projected difference quotients are given in \cite{Mengesha2012, KassmannMengeshaScott2019}. See also \cite{Temam-Miran} for continuous vector fields. 

\begin{lemma}\label{lma:ZeroSet}
    For $1\leq p<\infty$, suppose that $\bu \in L^p(B;\bbR^n)$ satisfies $\cE_1(\bu;B) = 0$. Then $\bu \in \cR$. 
\end{lemma}

\begin{proof}
The proof is similar to that of \cite[Lemma 2.8]{KassmannMengeshaScott2019}. 
Assume first that $\bu$ is continuous on $\overline{B}$. 
Since the integrand is non-negative and $|\bx-\by| < 2 < 4$ for all $\bx, \by \in B$, the condition $\cE_1(\bu;B) = 0$ implies that for every $(\bx, \by) \in B \times B$ with $|\bx-\by| > 1/8$,
\begin{equation}\label{eq:orthog}
    (\bu(\bx)-\bu(\by)) \cdot (\bx-\by) = 0, \qquad \forall \bx \in B \text{ and } \forall \by \in \Gamma(\bx):= \{ \by \in B : |\bx-\by| > 1/8 \}.
\end{equation}
Note that $\Gamma(\bx) = B \setminus \overline{B(\bx,1/8)}$ for all $\bx \in B$.
Now fix a parameter $\delta \in (0, 1/8)$ and let $\{\be_i\}_{i=1}^n$ denote the standard orthonormal basis of $\mathbb{R}^n$.
Fix $\bx_0 \in B(0,5/8)$.
Then $\overline{B(\bx_0,1/4)} \subset B$ and $\overline{B(\bx_0+\delta \be_i,1/8)} \subset B$ for each $i$.
Moreover, for any $i$ and for any $\bx \in \overline{B(\bx_0 + \delta \be_i,1/8)}$, $|\bx-\bx_0| \leq |\bx-\bx_0 - \delta \be_i| + \delta \leq 1/8 + \delta < 1/4$. Thus $\cup_{i=1}^n \overline{B(\bx_0 + \delta \be_i,1/8))} \subset \overline{B(\bx_0,1/4)}$, and so
\begin{equation*}
    B \setminus \overline{B(\bx_0,1/4)} \subset \wt{\Gamma}(\bx_0), \qquad \text{ where } \wt{\Gamma}(\bx_0) := \Gamma(\bx_0) \cap \bigcap_{i=1}^n \Gamma(\bx_0 + \delta \be_i).
\end{equation*}
For every $\bx \in \widetilde{\Gamma}(\bx_0)$, \eqref{eq:orthog} holds at both $\bx_0$ and $\bx_0 + \delta \be_i$:
\begin{align}
    (\bu(x) - \bu(\bx_0)) \cdot (\bx - \bx_0) &= 0, \label{eq:pt0} \\
    (\bu(\bx) - \bu(\bx_0 + \delta \be_i)) \cdot (\bx - \bx_0 - \delta \be_i) &= 0. \label{eq:pti}
\end{align}
Subtracting \eqref{eq:pt0} from \eqref{eq:pti} yields
\begin{equation}
    -\delta (\bu(x) - \bu(x_0)) \cdot \be_i + (\bu(x_0) - \bu(\bx_0 + \delta \be_i)) \cdot (\bx - \bx_0 - \delta \be_i) = 0.
\end{equation}
Rearranging terms, we obtain for each $i \in \{1, \dots, n\}$:
\begin{equation}\label{eq:coord-linear}
    \bu(\bx) \cdot \be_i = \frac{1}{\delta} \big( \bu(\bx_0 + \delta \be_i) - \bu(\bx_0) \big) \cdot (\bx - \bx_0) + \bu(\bx_0) \cdot \be_i.
\end{equation}
Because $\bu(\bx) \cdot \be_i$ are affine functions of $\bx$ on $\widetilde{\Gamma}(\bx_0)$ for each $i$, it follows that $\bu(\bx) = \bA \bx  + \bfb$ almost everywhere on $\widetilde{\Gamma}(\bx_0)$, hence on $B \setminus \overline{B(\bx_0,1/4)}$. 
To see that $\bu$ must be the same affine map on all of $B$, note that 
we can choose two points $\bx_0, \tilde{\bx}_0 \in B(0,5/8)$ such that $\overline{B(\bx_0,1/4)} \cap \overline{B(\tilde{\bx}_0,1/4)} = \emptyset$. Then it follows that $\wt{\Gamma}(\bx_0) \cap \wt{\Gamma}(\tilde{\bx}_0) \neq \emptyset$ and $\wt{\Gamma}(\bx_0) \cup \wt{\Gamma}(\tilde{\bx}_0) = B$. The nonempty intersection implies that $\bu$ must be the same affine map on both regions, hence on all of $B$. 
Finally, substituting $\bu(\bx) = \bA \bx + \bfb$ into \eqref{eq:orthog} and choosing $\by = 0$ gives
\begin{equation}
    (\bA \bx) \cdot \bx = \left( \frac{\bA+\bA^T}{2} \bx \right) \cdot \bx = 0, \qquad  \frac{1}{8} < |\bx| < 1.
\end{equation}
Therefore $\bA + \bA^T = 0$.

Now assume $\bu\in L^{p}(B; \mathbb{R}^n)$.
Let $\varphi$ be a standard radially symmetric $C^\infty$ mollifier with support in $B(0,1)$ and $\int_{B(0,1)} \varphi = 1$. For $\veps > 0$ sufficiently small, let $\varphi_\veps(\bx) = \veps^{-n} \varphi(\bx/\veps)$. Define $\bu_\veps = \varphi_\veps \ast \bu$, where the convolution is defined by extending $\bu$ outside of $B$ by $0$. Then by definition of $\bu_\veps$ and by the assumption \eqref{eq:orthog}
\begin{equation}\label{eq:orthog-lp}
    (\bu_\veps(\bx)-\bu_\veps(\by)) \cdot (\bx-\by) = 0 \qquad 
    \begin{gathered}
        \forall \bx \in B(0,1-\veps) \text{ and}  \\ \forall \by \in \Gamma_\veps(\bx) := \{ \by \in B(0,1-\veps) \, : \, |\bx-\by| > 1/8 \}.
    \end{gathered}
\end{equation}
We can now apply the previous argument to the smooth function $\bu_\veps$ (modifying choices of $\bx_0$ to account for the reduced radius $1-\veps$ and sets $\Gamma_\veps$) to conclude that $\bu_\veps$ must be an affine map on the set $B(0, 1-\veps)$.
Call this affine map $\bspsi_\veps$. 
Note that $\{\bspsi_\veps\}_\veps$ forms a bounded sequence in $\cR$; by \eqref{eq:AffineMapEquivalence:Rescaled},
\begin{equation*}
    \vnorm{\bspsi_\veps}_{\mathrm{aff}} \leq C \vnorm{\bspsi_\veps}_{L^p(B(0,1-\veps))} \leq C \vnorm{\bu_\veps}_{L^p(B)} \leq C \vnorm{\bu}_{L^p(B)},
\end{equation*}
where $C$ is independent of $\veps$.
Therefore since $\bspsi_\veps$ belongs to the finite-dimensional space $\cR$, there exists a subsequence (not relabeled) that converges in $\vnorm{\cdot}_{\mathrm{aff}}$ to an affine map $\bspsi$. Clearly $\bspsi \in \cR$ as well. Now, for a fixed $\veps_0$, we have for $\veps < \veps_0$ the identity $\bu_\veps = \bspsi_\veps$ on $B(0,1-\veps)$ and so
\begin{equation*}
    \begin{split}
        \vnorm{\bu - \bspsi}_{L^p(B(0,1-\veps_0))} &\leq \vnorm{\bu - \bu_\veps}_{L^p(B)} + \vnorm{\bspsi_\veps - \bspsi}_{L^p(B(0,1-\veps))} \\
        &\leq \vnorm{\bu - \bu_\veps}_{L^p(B)} + \vnorm{\bspsi_\veps - \bspsi}_{L^p(B)}.
    \end{split}
\end{equation*}
Since $\bu_\veps \to \bu$ in $L^p(B;\bbR^n)$ and since the $L^p(B;\bbR^n)$ norm is equivalent to the $\vnorm{\cdot}_{\mathrm{aff}}$ norm on $\cR$, we deduce that $\vnorm{\bu - \bspsi}_{L^p(B(0,1-\veps_0))} = 0$. The dominated convergence theorem as $\veps_0 \to 0$ implies that $\bu = \bspsi$ almost everywhere on $B$. 
\end{proof}

The next lemma shows the coercivity of an $L^p$ functional of the projected difference quotient, with no smoothness assumed on $\bu$.

\begin{lemma}\label{lma:CoercivityUnscaled}
For $1\leq p<\infty$, there exists $C_2 = C_2(p, n) > 0$ such that the following holds: for all $\bu \in L^p(B;\bbR^n)$
    \begin{equation*}
        \vnorm{ \bu - \Pi_B[\bu]}_{L^p(B)}^p \leq C_2 \cE_1(\bu;B).
    \end{equation*}
    As a consequence, for any $r > 0$ and $\bx_0 \in \bbR^n$,and $\bu \in L^p(B(\bx_0,r),\bbR^n)$, we obtain 
\begin{equation}\label{eq:CoercivityScaled:1}
     \|\bu-\Pi_{B(\bx_0,r)}[\bu]\|_{L^{p}(B(\bx_0,r))}^p  \leq C_2\, r^{-n} \cE_r(\bu;B(\bx_0,r)).
\end{equation}
\end{lemma}

\begin{proof}
    Suppose to the contrary; then there exists $\{\bu_k\} \subset L^p(B;\bbR^n)$ such that for each $k \in \bbN$
    \begin{equation*}
        \vnorm{\bu_k}_{L^p(B)} = 1, \qquad \Pi_B[\bu_k] = 0, \qquad \text{ with } \cE_1(\bu_k;B) \to 0 \text{ as } k \to \infty.
    \end{equation*}
    We will show that $\{\bu_k \}$ converges in $L^p(B;\bbR^n)$ and derive a contradiction.
    
   For $x\in B$, we define the symmetric matrix 
    \begin{equation*}
        \bM(\bx) = \int_B \bb1_{ \{ \frac{1}{8} \leq |\bx-\by| \leq 4\} } (\bx-\by) \otimes (\bx-\by) \, \rmd \by.
    \end{equation*}
    It then follows from Lemma~\ref{lma:PositiveFunctional} in the case $p = 2$ that there exists a constant $c > 0$, independent of $\bx$, such that
\begin{equation*}
    (\bM(\bx)\bsxi) \cdot \bsxi \geq \frac{1}{8} \Phi \left( \frac{\bsxi}{|\bsxi|}, \bx \right) |\bsxi|^2 \geq c |\bsxi|^2 \quad \text{for all } \bsxi \in \bbR^n \setminus \{\mathbf{0}\},
\end{equation*}
where $\Phi$ is the auxiliary functional defined in Lemma~\ref{lma:PositiveFunctional}. 
    Thus, $\bM$ is positive definite and so the inverse matrix $\bM^{-1} : B \to \bbR^{n \times n}$ exists with $\bM^{-1} \in L^\infty(B)$.
We now define the integral operator
    \begin{equation*}
        \bbK\bv(\bx) = \bM(\bx)^{-1} \int_B \bb1_{ \{ \frac{1}{8} \leq |\bx-\by| \leq 4 \} } \big( (\bx-\by) \otimes (\bx-\by) \big) \bv(\by) \, \rmd \by.
    \end{equation*}
    The operator $\bbK$ is a product of an $L^{\infty}$ function with a convolution operator.
    Thus, an argument similar to that of the Riesz-Frechet-Kolmogorov theorem \cite{Brezis2011} (see also the proof of \cite[Lemma 5.4]{Mengesha2012}) shows that $\bbK : L^p(B;\bbR^n) \to L^p(B;\bbR^n)$ is a compact operator. Moreover, by Minkowski's convolution inequality and the bound on $\bM^{-1}$
    \begin{equation*}
        \vnorm{\bv - \bbK \bv}_{L^p(B)}^p \leq C(n,p) \cE_1(\bv;B), \qquad \forall \bv \in L^p(B;\bbR^n).
    \end{equation*}
    Now, since $\vnorm{\bu_k}_{L^p(B)}$ is a bounded sequence, a subsequence (not relabeled) of $(\bbK \bu_k)$ converges strongly in $L^p(B;\bbR^n)$, say to a function $\bw$. Then
    \begin{equation*}
        \begin{split}
            \vnorm{\bu_k - \bw}_{L^p(B)}^p &\leq C(n,p) \vnorm{\bu_k - \bbK \bu_k}_{L^p(B)}^p + C(n,p) \vnorm{\bbK \bu_k - \bw }_{L^p(B)}^p \\
            &\leq C(n,p) \cE_1(\bu_k;B) + C(n,p) \vnorm{ \bbK \bu_k - \bw }_{L^p(B)}^p \\
            &\to 0 \text{ as } k \to \infty.
        \end{split}
    \end{equation*}
    Thus, $\bu_k \to \bw$ strongly in $L^p(B;\bbR^n)$. There exists a subsequence (again not relabeled) that converges to $\bw$ almost everywhere, and so by Fatou's lemma
    \begin{equation*}
        \cE_1(\bw;B) \leq \liminf_{k \to \infty} \cE_1(\bu_k;B) = 0.
    \end{equation*}
    Therefore, $\bw \in \cR$ by Lemma \ref{lma:ZeroSet}.  
    However, $\Pi_B$ is continuous on $L^p(B;\bbR^n)$ and fixes the elements of $\cR$, so
    \begin{equation*}
        \bw = \Pi_B[\bw] = \lim\limits_{k \to \infty} \Pi_B[\bu_k] = 0.
    \end{equation*}
    Thus, $\bw = 0$ almost everywhere. But $\vnorm{\bw}_{L^p(B)} = \lim\limits_{k \to \infty} \vnorm{\bu_k}_{L^p(B)} = 1$, a contradiction.

    Finally, \eqref{eq:CoercivityScaled:1} follows by applying the first part 
    to the function $\bu(\bx_0 + r \bx)$ and rescaling.  
\end{proof}

Without a projection operator that is continuous on $L^p$, the above proof would not be possible. Projections that appear in Korn's inequalities typically appear via minimization of the seminorms, and so are bounded only on Sobolev spaces. Fixing the projection explicitly will be used to our advantage many times in the subsequent arguments.

We will apply the scaled estimate \eqref{eq:CoercivityScaled:1} to an open cover of $B(0,1)$. The estimate needs to be adapted to treat sets that intersect the boundary of $B(0,1)$. This argument is similar to that of \cite{Ponce2004} used to establish a nonlocal Poincar\'e inequality.

\begin{lemma}\label{lma:CoercivityScaled}
    Let $\bx_0 \in \overline{B}$ and let $0 < r < 1$, so that $B((1-r)\bx_0,r) \subset B \cap B(\bx_0,2r)$. For $1\leq p<\infty)$, there exists $C=  C(n, p)>0$ such that for all $\bu \in L^p(B;\bbR^n)$
    \begin{equation}\label{eq:CoercivityScaled:Pf2}
    \vnorm{\bu - \Pi_{B((1-r)\bx_0,r)}[\bu] }^p_{L^p(B \cap B(\bx_0,2r))} \leq C r^{-n} \cE_r(\bu;B \cap B(\bx_0,2r)).
\end{equation}
\end{lemma}

\begin{proof}
    Write $\bv = \bu - \Pi_{B((1-r)\bx_0,r)}[\bu]$. Since
\begin{equation*}
    \left| \bv(\bx) \cdot \frac{\bx-\by}{|\bx-\by|} \right|^p \leq 2^p \left( \left| (\bv(\bx) - \bv(\by)) \cdot \frac{\bx-\by}{|\bx-\by|} \right|^p + \left| \bv(\by) \cdot \frac{\bx-\by}{|\bx-\by|} \right|^p \right),
\end{equation*}
integrating over $\by \in B((1-r)\bx_0,r) \cap \{ |\bx - \by| \geq r/8 \}$ gives
\begin{equation*}
    \begin{split}
        &\int_{B((1-r)\bx_0,r)} \bb1_{ \{ |\bx-\by| \geq r/8 \} } \left| \bv(\bx) \cdot \frac{\bx-\by}{|\bx-\by|} \right|^p \, \rmd \by \\
        \leq& C \int_{B((1-r)\bx_0,r)} \bb1_{ \{ |\bx-\by| \geq r/8 \} }  \left| (\bv(\bx) - \bv(\by)) \cdot \frac{\bx-\by}{|\bx-\by|} \right|^p \, \rmd \by \\
        &+ C \int_{B((1-r)\bx_0,r)} \bb1_{ \{ |\bx-\by| \geq r/8 \} } \left| \bv(\by) \cdot \frac{\bx-\by}{|\bx-\by|} \right|^p \, \rmd \by.
    \end{split}
\end{equation*}
After rescaling the integral, the left-hand side is seen to be equal to
\begin{equation*}
    r^n |\bv(\bx)|^p \Phi \left( \frac{\bv(\bx)}{|\bv(\bx)|}, \frac{\bx - (1-r) \bx_0}{r} \right), 
\end{equation*}
where the function $\Phi$ is as defined in Lemma \ref{lma:PositiveFunctional}.
Note that $\frac{\bx - (1-r) \bx_0}{r} \in \overline{B(0,3)}$ whenever $\bx_0 \in \overline{B}$ and $\bx \in B(\bx_0,2r)$. Thus, Lemma \ref{lma:PositiveFunctional} implies that there is a constant $C_1$ independent of $r$, $\bx_0$, and $\bu$ such that
\begin{equation*}
    \begin{split}
    C_1 r^n |\bv(\bx)|^p \leq  C \int_{B((1-r)\bx_0,r)} \bb1_{ \{ |\bx-\by| \geq r/8 \} }  \left| (\bv(\bx) - \bv(\by)) \cdot \frac{\bx-\by}{|\bx-\by|} \right|^p \rmd \by +  C \int_{B((1-r)\bx_0,r)} \left| \bv(\by) \right|^p \rmd \by.
    \end{split}
\end{equation*}
Integrating over $\bx \in B \cap B(\bx_0,2r)$ and using that $|\bx-\by| \leq 4r$ on the domain of integration, we obtain 
\begin{equation*}
    \int_{B \cap B(\bx_0,2r)} |\bv(\bx)|^p \, \rmd \bx \leq Cr^{-n} \cE_r(\bv;B \cap B(\bx_0,2r)) + C \frac{|B \cap B(\bx_0,2r)|}{r^n} \int_{B((1-r)\bx_0,r)} \left| \bv(\by) \right|^p \, \rmd \by.
\end{equation*}
Finally, we recall that $\bv = \bu - \Pi_{B((1-r)\bx_0,r)}[\bu]$ and use \eqref{eq:CoercivityScaled:1} on the last integral to get the result.
\end{proof}

Building on the previous coercivity estimates, the following lemma compares the nonlocal functionals $\cD_r(\bu; B)$ and $\cE_r(\bu; B)$.

\begin{lemma}\label{lma:CoercivityScaled:Korn1st}For $1<p < \infty$
    There exists $C = C(n,p) > 0$ such that the following holds: for all $\bu \in L^p(B;\bbR^n)$ and for each $0 < r < 1$ there exists $\bspsi \in \cR$ such that 
    \begin{equation}\label{eq:CoercivityScaled:Korn1st}
        \cD_r(\bu- \bspsi;B) + r^{n+p} \vnorm{\bu - \bspsi}_{L^p(B)}^p \leq C(n,p) \cE_r(\bu;B)
    \end{equation}
    
\end{lemma}

\begin{proof}
For $r > 0$, there exists an integer $N=N(r)$ and points $\{ \bx_i \}_{i=1}^{N} \subset \overline{B}$ such that $\overline{B} \subset \cup_{i=1}^N B(\bx_i,r/4)$ and the enlarged cover is locally finite: $\sum_{i=1}^N \bb1_{B(\bx_i,2r)}(x) \leq C(n)$.
Let $\varphi_i$ be a $C^\infty$ partition of unity satisfying $0 \leq \varphi_i \leq 1$, $\supp \varphi_i \subset B(\bx_i,r)$, and $|\grad \varphi_i| \leq Cr^{-1}$.
For the sets $B_i = B \cap B(\bx_i,2r)$, define the infinitesimal rigid motion $\bspsi_i = \Pi_{B((1-r)\bx_i,r))}[\bu]$ so that \eqref{eq:CoercivityScaled:Pf2} is satisfied on $B_i$. Define the function
\begin{equation*}
    \bv = \sum_{i=1}^N \varphi_i \bspsi_i.
\end{equation*}
This function will serve as the global approximation of $\bu$ on $B$. Note that $\bv \in C^\infty(\overline{B})$. Let $\bspsi \in \cR$ satisfy
    \begin{equation}\label{eq:CoercivityScaled:Korn1st:Pf1}
        \vnorm{\bv - \bspsi}_{W^{1,p}(B)} \leq C(n,p) \vnorm{e(\bv)}_{L^p(B)}^p;
    \end{equation}
    such a $\bspsi$ exists because the classical Korn's first unconstrained inequality holds on $B$.

We begin the estimate with
\begin{equation}\label{eq:CoercivityScaled:Pf:3}
    \cD_r(\bu-\bspsi;B) \leq C(p) (\cD_r(\bu-\bv;B) + \cD_r(\bv-\bspsi;B)).
\end{equation}
and estimate the terms in the right hand side separately. Let us estimate the first term as  
\begin{equation}\label{eq:CoercivityScaled:Pf:4}
\begin{split}
    \cD_r(\bu-\bv;B) &\leq 2^{p-1} \int_B \int_B \mathds{1}_{ \{ r/2 \leq |\bx-\by| \leq r \} } (|\bu(\bx)-\bv(\bx)|^p + |\bu(\by)-\bv(\by)|^p )\, \rmd \by \, \rmd \bx \\
    &\leq C(n,p) r^n \vnorm{\bu-\bv}_{L^p(B)}^p.
\end{split}
\end{equation}
Now, $\bu-\bv = \sum_i \varphi_i(\bu - \bspsi_i)$, and Jensen's inequality along with \eqref{eq:CoercivityScaled:Pf2} implies
\begin{equation}\label{eq:CoercivityScaled:Pf7}
   \vnorm{\bu-\bv}_{L^p(B)}^p \leq C(n)^{p-1}\sum_i \int_{B_i} |\bu - \bspsi_i|^p \, \rmd \bx \leq C r^{-n} \sum_i \cE_r(\bu;B_i) \leq C(n,p) r^{-n} \cE_r(\bu;B).
\end{equation}
The last inequality follows from the local finiteness of the cover, as well as the fact that $|\bx-\by|$ is comparable to $\diam B_i$ in $\cE_r(\bu;B_i)$.

To estimate the second term in the right hand side of \eqref{eq:CoercivityScaled:Pf:3}, we set $\bw = \bv - \bspsi \in C^\infty$.  Taylor expanding  and using Jensen's inequality, we get 
\begin{equation*}
    \begin{split}
        |\bw(\bx) - \bw(\by)|^p \leq |\bx-\by|^p \int_0^1 |\grad \bw(t\bx + (1-t)\by)|^p \, \rmd t.
    \end{split}
\end{equation*}
Then changing coordinates and extending $\grad \bw$ outside $B$ by $0$ if necessary, we have
\begin{equation}\label{eq:CoercivityScaled:Pf:5}
\begin{split}
 \cD_r(\bv-\bspsi;B) =  \cD_r(\bw;B) &\leq \int_B \int_{B(0,2)} \bb1_{ \{ 1/2 \leq |\bh|/r \leq 1 \} }\int_0^1 |\grad \bw(\bx + t \bh)|^p |\bh|^{p} \, \rmd t \, \rmd \bx \, \rmd \bh \\
    &\leq C(n,p) r^{n+p} \vnorm{\grad \bw}_{L^p(B)}^p\\
    &= C(n,p) r^{n+p} \vnorm{\grad \bv - \grad\bspsi}_{L^p(B)}^p
\end{split}
\end{equation}
We estimate the right hand side of \eqref{eq:CoercivityScaled:Pf:5} further. From the choice of $\bspsi$ in  \eqref{eq:CoercivityScaled:Korn1st:Pf1} and by Korn's first unconstrained inequality  for $1 < p < \infty$ applied to $\bv-\bspsi$ on $B$ gives
\begin{equation}\label{eq:CoercivityScaled:Pf8}
\vnorm{\grad \bv-\grad \bspsi}_{L^p(B)}^p \leq C(n,p) \vnorm{e(\bv)}_{L^p(B)}^p.
\end{equation}
We directly estimate $e(\bv)$ using the formula;
\begin{equation*}
    \grad \bv = \sum_i \varphi_i \grad \bspsi_i + \bspsi_i \otimes \grad \varphi_i = \sum_i \varphi_i \grad \bspsi_i + (\bspsi_i - \bu) \otimes \grad \varphi_i,
\end{equation*}
where the latter term appears because $\sum_i \nabla \varphi_i  = 0$.
Now since $\bspsi \in \cR$,
\begin{equation*}
    e(\bv) = \frac{1}{2} \sum_i (\bspsi_i - \bu) \otimes \grad \varphi_i + \grad \varphi_i \otimes (\bspsi_i - \bu). 
\end{equation*}
Therefore, using the bounds on $\grad \varphi_i$ we have,  similarly to \eqref{eq:CoercivityScaled:Pf7}, 
\begin{equation}\label{eq:CoercivityScaled:Pf9}
    \vnorm{e(\bv)}_{L^p(B)}^p \leq C r^{-p} \sum_i \int_{B_i} |\bu - \bspsi_i|^p \, \rmd \bx \leq C r^{-n-p} \cE_r(\bu;B).
\end{equation}
Then \eqref{eq:CoercivityScaled:Pf9}  
applied to \eqref{eq:CoercivityScaled:Pf8} gives the estimate
\begin{equation}\label{eq:CoercivityScaled:11}
\begin{split}
   \vnorm{\grad \bv-\grad \bspsi}_{L^p(B)}^p
    &\leq  C r^{-n-p} \cE_r(\bu;B).
\end{split}
\end{equation}
Therefore,  
\eqref{eq:CoercivityScaled:Pf7} inserted in \eqref{eq:CoercivityScaled:Pf:4},  and  \eqref{eq:CoercivityScaled:11} inserted into \eqref{eq:CoercivityScaled:Pf:5}  gives
\begin{equation}\label{eq:CoercivityScaled:Korn1st:Pf3}
    \cD_r(\bu-\bspsi;B) \leq C \cE_r(\bu;B).
\end{equation}
Finally, we use \eqref{eq:CoercivityScaled:Pf7} and the fact that $r < 1$, as well as \eqref{eq:CoercivityScaled:Korn1st:Pf1}, to obtain
\begin{equation}\label{eq:CoercivityScaled:Korn1st:Pf4}
    r^{n+p} \vnorm{\bu - \bspsi}_{L^p(B)}^p \leq C r^{n+p} \vnorm{\bu - \bv}_{L^p(B)}^p + C r^{n+p} \vnorm{\bv - \bspsi}_{L^p(B)}^p \leq C \cE_r(\bu;B).
\end{equation}
Collecting \eqref{eq:CoercivityScaled:Korn1st:Pf3} and \eqref{eq:CoercivityScaled:Korn1st:Pf4} gives the result.
\end{proof}

Note that in the previous theorem $\bspsi$ may depend on $r$, since $\bv$ does. We next show that we can choose $\psi$ as the projection $\Pi_B$, thus providing a rigid motion independent of $r$.

\begin{corollary}\label{cor:CoercivityScaled:Korn1st}
    There exists $C(n,p) > 0$ such that the following holds: for all $\bu \in L^p(B;\bbR^n)$
    \begin{equation*}
        \cD_r(\bu- \Pi_B[\bu];B) + r^{n+p} \vnorm{\bu - \Pi_B[\bu]}_{L^p(B)}^p \leq C(n,p) \cE_r(\bu;B)
    \end{equation*}
    for each $0 < r < 1$.
\end{corollary}

\begin{proof}
    For each $r$, let $\bspsi_r \in \cR$ be as in the previous lemma. Then by the $L^{p}$-continuity of $\Pi_B$ given in  \eqref{eq:ProjContLp}
    \begin{equation*}
    \begin{split}
        \vnorm{\bu - \Pi_B[\bu]}_{L^p(B)} &\leq \vnorm{\bu - \bspsi_r}_{L^p(B)} + \vnorm{\Pi_B[\bu] - \bspsi_r}_{L^p(B)} \\
        &= \vnorm{\bu - \bspsi_r}_{L^p(B)} + \vnorm{\Pi_B[\bu - \bspsi_r]}_{L^p(B)} \\
        &\leq (1 + C(n,p))\vnorm{\bu - \bspsi_r}_{L^p(B)},
    \end{split}
    \end{equation*}
    and by \eqref{eq:CoercivityScaled:Pf:5} (applied to $\bw = \Pi_B[\bu - \bspsi_r]$) and \eqref{eq:AffineMapEquivalence:Rescaled}
    \begin{equation*}
    \begin{split}
        \cD_r(\bu - \Pi_B[\bu];B) &\leq C(p) \cD_r(\bu - \bspsi_r;B) + C(p) \cD_r(\Pi_B[\bu] - \bspsi_r;B) \\
        &= C(p) \cD_r(\bu - \bspsi_r;B) + C(p) \cD_r(\Pi_B[\bu - \bspsi_r];B) \\
        &\leq C(n,p) \left( \cD_r(\bu - \bspsi_r;B) + r^{n+p} |\Sigma(\bu-\bspsi_r)|^p \right) \\
        &\leq C(n,p) ( \cD_r(\bu - \bspsi_r;B) + r^{n+p} \vnorm{\bu-\bspsi_r}_{L^p(B)}^p ).
    \end{split}
    \end{equation*}
    Therefore,
    \begin{equation*}
        \cD_r(\bu - \Pi_B[\bu];B) + r^{n+p} \vnorm{\bu - \Pi_B[\bu]}_{L^p(B)}^p \leq C(n,p) \left( \cD_r(\bu - \bspsi_r;B) + r^{n+p} \vnorm{\bu-\bspsi_r}_{L^p(B)}^p  \right),
    \end{equation*}
    and the result follows from the previous lemma.
\end{proof}

Finally, we state and prove  the fractional Korn's inequality for balls as follows.  
\begin{theorem}\label{thm:Korns1stOnBall}
    There exists $C(n,p) > 0$ such that for all $\bu \in \cX^{s,p}(B)$
    \begin{equation}
        \vnorm{\bu - \Pi_B[\bu]}_{L^p(B)}^p + (1-s)[\bu - \Pi_B[\bu]]_{W^{s,p}(B)}^p \leq C(n,p) (1-s) [\bu]_{\cX^{s,p}(B)}^p.
    \end{equation}
\end{theorem}

\begin{proof}
 For each $k \ge -1$, we decompose the product domain $B\times B$ into the dyadic annuli  as $B \times B  = \bigcup_{k=-1}^\infty A_k$ where 
  \[
    A_k=\left\{ (\bx, \by) \in B \times B : 2^{-k-1} \le |\bx - \by| \le 2^{-k} \right\}.
  \]
  Then there exists a constant $C$ independent of $k$ such that
    \begin{equation*}
        \begin{split}
            [\bu-\Pi_B[\bu]]_{W^{s,p}(B)}^p \leq C(n,p) \sum_{k = -1}^\infty (2^{-k})^{-n-sp} \cD_{2^{-k}}(\bu - \Pi_B[\bu];B).
        \end{split}
    \end{equation*}
  where we used the fact that on $D_k$, the singular kernel satisfies the two-sided bound:
  \[
    2^{k(n+ps)} \le \frac{1}{|\bx - \by|^{n+ps}} \le 2^{(k+1)(n+ps)}.
  \]
    Next, we use the bounds $\frac{1}{1-s} \frac{1}{p \ln(2)} \leq \sum_{k=0}^\infty 2^{-k(p-sp)} \leq \frac{2^p}{2^p-1} \frac{1}{1-s}$, valid for all $s \in (0,1)$, to write
    \begin{equation*}
        \frac{1}{1-s} \vnorm{\bu-\Pi_B[\bu]}_{L^p(B)}^p \leq C(p) \sum_{k = -1}^\infty 2^{-k(p-sp)} \vnorm{\bu-\Pi_B[\bu]}_{L^p(B)}^p.
    \end{equation*}
    The $k = -1$ and $k = 0$ terms in both series are bounded from above by 
    $C(n,p) \vnorm{\bu-\Pi_B[\bu]}_{L^p(B)}^p$, which is in turn controlled by $C(n,p) \cE_1(\bu;B)$ thanks to Lemma \ref{lma:CoercivityUnscaled}. Thus we can apply Lemma \ref{cor:CoercivityScaled:Korn1st} for all other terms in both series to obtain a constant $C = C(n,p)$ independent of $k$ such that
    \begin{equation*}
        \begin{split}
            [\bu-\Pi_B[\bu]]_{W^{s,p}(\Omega)}^p &+ \frac{1}{1-s} \vnorm{\bu-\Pi_B[\bu]}_{L^p(B)}^p  \\
            \leq& C(n,p) \sum_{k = -1}^\infty (2^{-k})^{-n-sp} \left( \cD_{2^{-k}}(\bu - \Pi_B[\bu];B) + (2^{-k})^{n+p} \vnorm{\bu-\Pi_B[\bu]}_{L^p(B)}^p \right) \\
            \leq& C(n,p) \sum_{k = -1}^\infty (2^{-k})^{-n-sp} \cE_{2^{-k}}(\bu;B).
        \end{split}
    \end{equation*}
    Even though the annuli are enlarged in the definition of $\cE_{2^{-k}}(\bu;B)$, the sets $ \{ (\bx,\by) \in B \times B \, : \, 2^{-k-3} \leq |\bx-\by| \leq 2^{-k+2} \} $ still form a locally finite cover of $B \times B$. Moreover, $|\bx-\by| \approx 2^{-k}$ still holds on each domain, so we have
    \begin{equation*}
        \sum_{k = -1}^\infty (2^{-k})^{-n-sp} \cE_{2^{-k}}(\bu;B) \leq C(n,p) [\bu]_{\cX^{s,p}(B)}^p,
    \end{equation*}
    as desired.
\end{proof}

After rescaling, we also have the dependence of the constants on the diameter of a ball.
\begin{theorem}\label{thm:KornBallStrong}There exists $C(n,p) > 0$ such that 
    for any $\bx_0 \in \bbR^n$, $R > 0$, $B_R = B(\bx_0,R)$, and $\bu \in \cX^{s,p}(B_R)$    
    \begin{equation}
        R^{-sp} \vnorm{\bu - \Pi_{B_R}[\bu]}_{L^p(B_R)}^p + (1-s)[\bu - \Pi_{B_R}[\bu]]_{W^{s,p}(B_R)}^p \leq C(n,p) (1-s) [\bu]_{\cX^{s,p}(B_R)}^p.
    \end{equation}
\end{theorem}

\subsection{The case \texorpdfstring{$p=1$}{p=1}}\label{subsec:1}

In this section, we prove the following:

\begin{theorem}\label{thm:Korns1stOnBall:L1}
    Let $B = B(0,1)$ and let $0 < s < 1$. Then there exists $C = C(n) > 0$ such that for all $\bu \in \cX^{s,1}(B;\bbR^n)$
    \begin{equation}\label{eq:Korns1stOnBall:L1}
        \vnorm{\bu - \Pi_B[\bu]}_{L^1(B)} + (1-s)[\bu - \Pi_B[\bu]]_{W^{s,1}(B)} \leq C [\bu]_{\cX^{s,1}(B)}.
    \end{equation}
\end{theorem}
Although \eqref{eq:CoercivityScaled:Pf2} holds for $p = 1$, we cannot make use of an argument like that in the proof of Lemma \ref{lma:CoercivityScaled:Korn1st}, since in the regime $p = 1$ we cannot make use of any classical local Korn inequality.  In fact, an estimate of the form \eqref{eq:CoercivityScaled:Korn1st} cannot hold when $p = 1$; multiplying \eqref{eq:CoercivityScaled:Korn1st} by $r^{-n-p}$ and taking $r \to 0$ formally recovers a classical Korn's first inequality statement for $L^1$, a contradiction. 
Our proof of \eqref{eq:Korns1stOnBall:L1} requires a genuinely nonlocal estimate incorporating all dyadic scales of $r$ in the functionals $\cE_r$, instead of the single scale present in \eqref{eq:CoercivityScaled:Korn1st}.

\begin{lemma}\label{lma:CoercivityScaled:Korn1st:L1}
    There exists $C(n) > 0$ such that for all $\bu \in L^1(B;\bbR^n)$ and for all $k \in \bbN_0$
    \begin{equation}\label{eq:CoercivityScaled:Korn1st:L1}
        \cD_{2^{-k}}(\bu- \Pi_B[\bu];B) + 2^{-k(n+1)} \vnorm{\bu - \Pi_B[\bu]}_{L^1(B)} \leq C \sum_{j=0}^k \frac{2^{-k(n+1)}}{2^{-j(n+1)}} \cE_{2^{-j}}(\bu;B).
    \end{equation}
\end{lemma}

The proof of the lemma relies on the following covering of the unit ball $B$.  
Let us denote by $X_k$ the grid vertices; that is,  $X_k \subset \bbR^n$ are all vectors whose entries are $a2^{-k-2}$ with $a \in \bbZ$. We now place a ball of radius $r_k = 2^{-k}$ around each $\bx_k \in X_k\cap B$.
Set $\cF_0 = \{ B \}$, and for $k \geq 1$ take $\cF_k = \{ B(\bx,2^{-k}) \cap B \, : \, \bx \in X_k \cap B \}$. It is then clear that $\cF_k$  is an open cover of $B$. 
Moreover, for every  $U \in \cF_k$, $\frac{1}{4} \cdot 2^{-k} \leq \diam U \leq 4 \cdot 2^{-k}$, and there exists $x_U \in B$ such that \[B(x_U, \frac{1}{4} \cdot 2^{-k}) \subset U \subset B(x_U, 4 \cdot 2^{-k}).\]
We say that $U \in \cF_k$ has a parent $\wt{U} \in \cF_{k-1}$ if there exists $\wt{\bx} \in B$ such that $B(\wt{\bx},\frac{1}{4} \cdot 2^{-k}) \subset U \cap \wt{U}$. By construction, it is not difficult to see that every $U \in \cF_k$ has at least one parent $\wt{U} \in \cF_{k-1}$.
This covering satisfies additional properties which we summarize below. 

\begin{lemma}\label{lma:L1Cover}
    There exist dimensional constants $C > 0$ and $M_0 \geq 1$ such that the following properties of  the collection of the open coverings $\{\cF_k\}_{k\in \mathbb{N}_0}$ holds:
    \begin{enumerate}

        \item[{\bf i)}] For every $k \in \bbN_0$, the dilated collection $\{ M_0 U \, : \, U \in \cF_k \}$ is locally finite, and each point of $B(0,1)$ is contained in at most $C M_0^n$ of the elements of $\cF_k$.

        \item[{\bf ii)}] $\{ (\bx,\by) \in B \times B \, : \, 2^{-k-1} < |\bx-\by| < 2^{-k} \} \subset \bigcup_{U \in \cF_k} M_0 U \times M_0 U$.

        \item[{\bf iii)}] For $k > j$, a sequence $\{ U_j, U_{j+1}, \ldots, U_k \}$ is a nested chain between $U \in \cF_k$ and $\wt{U} \in \cF_j$ if $U_j = \wt{U}$, $U_k = U$, $U_\ell \in \cF_\ell$ for $\ell = j,j+1,\ldots,k$ and $U_{\ell - 1}$ is a parent of $U_\ell$ for $\ell = j+1,\ldots,k$. For any $\wt{U} \in \cF_j$, the set $\cM_{j,k}(\wt{U}) = \{ U \in \cF_k \, : \, \exists \text{ a nested chain between } U \text{ and } \wt{U} \}$ has cardinality at most $C M_0^n 2^{n(k-j)}$.
    \end{enumerate}
\end{lemma}

\begin{proof}
    Set $M_0 = 2$. 
  Then, for the collection of the open cover $\cF_k$ outlined above,  properties i) and ii) are readily verified with $C = 10^n$, where we are denoting $M_0 U = B(\bx,M_02^{-k}) \cap B$. To see property iii), 
    let $k > j$ and suppose there exists a nested chain between $B(\bx_k,2^{-k})$ and $B(\bx_j,2^{-j})$. Let $B(\bx_\ell,2^{-\ell})$, $\ell = j,\ldots,k$, denote the chain. For each $\ell$, property ii) implies there exists $B(\bz,2^{-\ell-2}) \subset B_\ell \cap B_{\ell-1}$, and so
    \begin{equation*}
        |\bx_\ell - \bz| \leq 2^{-\ell} - 2^{-\ell-2} = \frac{3}{4}2^{-\ell} \qquad \text{ and } \qquad |\bx_{\ell-1} - \bz| \leq 2^{-\ell+1} - 2^{-\ell-2} = \frac{7}{4} 2^{-\ell}.
    \end{equation*}
    Thus $|\bx_\ell - \bx_{\ell-1}| \leq \frac{5}{2} 2^{-\ell}$, and so
    $
        |\bx_j - \bx_k| \leq \sum_{\ell = j+1}^{k} |\bx_{\ell-1} - \bx_{\ell}| \leq \frac{5}{2} (2^{-j} - 2^{-k}).
    $
    Both $2^{k+2} \bx_j$ and $2^{k+2} \bx_k$ can be written as vectors with integer entries, thus there exists $\ba \in \bbZ^n$ that satisfies
    \begin{equation*}
        |\ba| = 2^{k+2} |\bx_j-\bx_k| \leq 10(2^{k-j} - 1).
    \end{equation*}
    To see that property iii) follows from this bound for arbitrary $\bx_k$ and $\bx_j$, we observe that for any $R > 0$ the cardinality of the set $\{ \ba \in \bbZ^n \, : \, |\ba| \leq R\}$ is bounded by $(2R + 1)^n$.
\end{proof}

\begin{proof}[Proof of Lemma \ref{lma:CoercivityScaled:Korn1st:L1}]
    We denote $\Pi_B[\bu] = \bspsi_0 = \bA \bx + \bfb$; Lemma \ref{lma:CoercivityUnscaled}
    implies
    \begin{equation}
        \vnorm{\bu-\bspsi_0}_{L^1(B)} \leq C(n) \cE_{1}(\bu;B).
    \end{equation}
    Thus the inequality for the $L^1$ norm is established, and we only need to estimate $\cD_{2^{-k}}(\bu-\bspsi_0;B)$.
    Fix $k \in \bbN_0$. By property ii) of the cover in Lemma \ref{lma:L1Cover},
    \begin{equation}\label{eq:CoercivityScaled:Korn1st:L1:Pf2}
        \cD_{2^{-k}}(\bu-\bspsi_0;B) \leq C(n) 
        \sum_{U \in \cF_k} \cD_{2^{-k}}(\bu-\bspsi_0;M_0 U).
    \end{equation}
    where used the local finiteness of the dilations as stated in property i) of Lemma \ref{lma:L1Cover}. 
    
    Now fix $U \in \cF_k$. By property iii) of Lemma \ref{lma:L1Cover} there exists a nested chain $\{U_0,\ldots,U_k\}$ between $U_k = U$ and $U_0 = B$.
    Apply Lemma \ref{lma:CoercivityScaled} to each set $U_j$ to obtain $\bspsi_j = \bA_j \bx + \bfb_j \in \cR$ satisfying
    \begin{equation}\label{eq:CoercivityScaled:Korn1st:L1:Pf3}
        \vnorm{\bu - \bspsi_j}_{ L^1(M_0 U_j) } \leq C(n) 2^{jn} \cE_{2^{-j}}(\bu;M_0 U_j), \quad j = 1,\ldots, k.
    \end{equation}
    Then 
    \begin{equation*}
        \begin{split}
            |\bu(\bx)-\bspsi_0(\bx) - (\bu(\by) - \bspsi_0(\by))| &\leq |\bu(\bx) - \bspsi_k(\bx)| + |\bu(\by) - \bspsi_k(\by)| \\
            &\qquad + \sum_{j=0}^{k-1} |\bspsi_{j+1}(\bx-\by) - \bspsi_j(\bx-\by)|.
        \end{split}
    \end{equation*}
    We integrate this inequality over $M_0 U \times M_0 U$. First, since by construction every $U\in \cF$ has a parent $\tilde{U}\in \cF_{k-1}$, by \eqref{eq:AffineMapEquivalence:Rescaled:MatrixPart} and
    \eqref{eq:CoercivityScaled:Korn1st:L1:Pf3}
    \begin{equation*}
        \begin{split}
            \int_{M_0 U} \int_{M_0 U} &\mathds{1}_{ \{ 2^{-k-1} < |\bx-\by| < 2^{-k} \} } |\bspsi_{j+1}(\bx-\by) - \bspsi_j(\bx-\by)| \, \rmd \by \, \rmd \bx \\
            &\leq |\bA_{j+1} - \bA_j| \int_{M_0 U} \, \rmd \bx \int_{\{|\bz| \leq C(n) 2^{-k}\}} |\bz| \, \rmd \bz \\
            &\leq C(n) 2^{-k(2n+1)+j(n+1)} \vnorm{\bspsi_{j+1} - \bspsi_j}_{L^1(B(\wt{\bx}, 2^{-j} /c))} \\
            &\leq C(n) 2^{-k(2n+1)+j(n+1)} \big( \vnorm{\bu - \bspsi_{j+1}}_{L^1(U_{j+1} \cap U_j)} + \vnorm{\bu - \bspsi_j}_{L^1(U_{j+1} \cap U_j)} \big) \\
            &\leq C(n) 2^{(j-k)(2n+1)} \big( \cE_{2^{-j-1}}(\bu; M_0 U_{j+1}) + \cE_{2^{-j}}(\bu; M_0 U_{j}) \big).
        \end{split}
    \end{equation*}
    The integrations of $|\bu(\bx) - \bspsi_k(\bx)|$ and $|\bu(\by) - \bspsi_k(\by)|$ are bounded in the same way using \eqref{eq:CoercivityScaled:Korn1st:L1:Pf3}, with $j = k$. 
    Therefore, for each $U \in \cF_k$ there exists a nested chain $B = U_0, \ldots, U_{k-1}, U_k = U$ such that
    \begin{equation}
        \cD_{2^{-k}}(\bu - \bspsi_0;M_0 U) \leq C(n) \sum_{j=0}^k 2^{(j-k)(2n+1)}\cE_{2^{-j}}(\bu; M_0 U_j).
    \end{equation}
    To sum this over all $U \in \cF_k$, we use property iii) of Lemma \ref{lma:L1Cover}.
    Let $j \in \{ 0,\ldots,k \}$; property v)  states that each set $\wt{U} \in \cF_j$ can appear as a parent in at most $C(n) 2^{n(k-j)}$ nested chains between $B$ and any element of $\cF_k$. Therefore, the number of times the term $\cE_{2^{-j}}(\bu;M_0 \wt{U})$ can appear in the sum is bounded by the same number. Choosing a nested chain for each $U \in \cF_k$, we have
    \begin{equation*}
        \sum_{U \in \cF_k} \cD_{2^{-k}}(\bu - \bspsi_0;M_0 U) \leq C(n) \sum_{j = 0}^k \sum_{\wt{U} \in \cF_j} 2^{n(k-j)} 2^{(j-k)(2n+1)}\cE_{2^{-j}}(\bu; M_0 \wt{U}).
    \end{equation*}   
    Then by finite overlap of the dilated collection, property i) of Lemma \ref{lma:L1Cover}, there exists $C(n)$ such that
    \begin{equation*}
        \sum_{\wt{U} \in \cF_j} \cE_{2^{-j}}(\bu;M_0\wt{U}) \leq C(n) \cE_{2^{-j}}(\bu;B), \qquad \forall j \in \bbN_0.
    \end{equation*}
    The conclusion of the theorem follows from \eqref{eq:CoercivityScaled:Korn1st:L1:Pf2} and the previous two displays.
\end{proof}

For a smooth function $\bu$, divide by $2^{-k(n+1)}$ in \eqref{eq:CoercivityScaled:Korn1st:L1} and take $k \to \infty$. The left-hand side converges to $\vnorm{\bu - \Pi_B[\bu]}_{W^{1,1}(B)}$, while the right-hand side is comparable to $k [e(\bu)]_{L^1(B)}$ for large $k$, and thus becomes unbounded.
Thus in the local regime this estimate remains consistent with the lack of the $L^1$ Korn's inequality. However, the fractional order $s$ of differentiability will ensure that the sum over all dyadic scales in the Gagliardo seminorm remains finite.

\begin{proof}[Proof of Theorem \ref{thm:Korns1stOnBall:L1}]
    The argument is similar to the proof of Theorem \ref{thm:Korns1stOnBall}. There exists $C(n)$ independent of $k$ such that
    \begin{equation*}
        \begin{split}
            \frac{1}{1-s} \vnorm{\bu - \Pi_B[\bu]}_{L^1(B)} &+ [\bu- \Pi_B[\bu]]_{W^{s,1}(B)} \\
            \leq& C(n) \sum_{k=-1}^\infty (2^{-k})^{-n-s} \left( \cD_{2^{-k}}(\bu- \Pi_B[\bu];B) + 2^{-k(n+1)} \vnorm{\bu - \Pi_B[\bu]}_{L^1(B)} \right).
        \end{split}
    \end{equation*}
    Then by Lemma \ref{lma:CoercivityScaled:Korn1st:L1} 
    \begin{equation*}
        \frac{1}{1-s} \vnorm{\bu - \Pi_B[\bu]}_{L^1(B)} + [\bu- \Pi_B[\bu]]_{W^{s,1}(B)} \leq C \sum_{k=-1}^\infty (2^{-k})^{-n-s} \sum_{j=0}^k 2^{(j-k)(n+1)} \cE_{2^{-j}}(\bu;B).
    \end{equation*}
    Using Tonelli's theorem to interchange the order of summation, we conclude
    \begin{equation*}
        \begin{split}
            \frac{1}{1-s} \vnorm{\bu - \Pi_B[\bu]}_{L^1(B)} + [\bu- \Pi_B[\bu]]_{W^{s,1}(B)} 
            &\leq C \sum_{j=-1}^\infty  (2^{j})^{n+s} \sum_{\ell=0}^\infty 2^{-\ell(1-s)} \cE_{2^{-j}}(\bu;B) \\
            &\leq  \frac{C(n) }{1-2^{-(1-s)}}  \sum_{j=-1}^\infty  (2^{j})^{n+s} \cE_{2^{-j}}(\bu;B)  \\
            &\leq \frac{C(n)}{1-s} [\bu]_{\cX^{s,1}(B)}
        \end{split}
    \end{equation*}
where the last inequality follows analogously to the comparable estimate in the proof of Theorem \ref{thm:Korns1stOnBall}. We additionally used the inequality $(1 - s) \le 2 - 2^s \le (2 \ln 2)(1 - s)$ valid for all $ s \in [0, 1].$
\end{proof}

The rescaled estimate (presented here without proof) is 
\begin{theorem}\label{thm:KornBallStrong:L1}
    Let $\bx_0 \in \bbR^n$, $R > 0$, and $B_{R} = B(\bx_0,R)$. There exists $C(n) > 0$ such that for all $\bu \in \cX^{s,1}(B_{R})$    
    \begin{equation}
        R^{-s} \vnorm{\bu - \Pi_{B_{R}}[\bu]}_{L^1(B)} + (1-s)[\bu - \Pi_{B_{R}}[\bu]]_{W^{s,1}(B_{R})} \leq C [\bu]_{\cX^{s,1}(B_{R})}.
    \end{equation}
\end{theorem}

\subsection{The case \texorpdfstring{$p = \infty$}{p=infinity}}\label{subsec:infty}

In this section, we prove a version of Korn's second inequality for H\"older spaces. We write $C^s(\Omega;\bbR^n) = W^{s,\infty}(\Omega;\bbR^n)$. 
We also write $\cK^s(\Omega) = \cX^{s,\infty}(\Omega)$ and $\cK^s(\Omega)$ is the space of continuous functions such that the norm $\vnorm{\bu}_{\cK^s(\Omega)} := \vnorm{\bu}_{L^\infty(\Omega)} + [\bu]_{\cK^s(\Omega)}$ is finite, where
\begin{equation*}
    [\bu]_{\cK^s(\Omega)} := \sup_{\substack{\bx,\by \in \overline{\Omega} \\ \bx \neq \by }} \frac{ \left| (\bu(\bx) - \bu(\by)) \cdot (\bx-\by) \right| }{ |\bx-\by|^{s+1} }.
\end{equation*}

\begin{lemma}\label{lma:Unscaled:Infty}
    Let $G \subset \bbR^n$ satisfy $B(\bx_0,r/2) \subset G \subset B(\bx_0,2r)$ for some $\bx_0 \in \bbR^n$, $0 < r < 1$.
    Then there exists $C = C(n) > 0$ such that the following holds:
    For all $\bu \in \cK^s(G)$ there exists $\bspsi \in \cR$ such that
    \begin{equation*}
        \vnorm{\bu - \bspsi}_{L^\infty(G)} \leq C r^s [\bu]_{\cK^s(G)}.
    \end{equation*}
\end{lemma}

\begin{proof}
    Set $M = [\bu]_{\cK^s(G)}$.
    Define the points $\bx_i = \bx_0 + \frac{r}{4}\be_i$, $i = 1,\ldots,n$, and define the affine vector field $\bsvarphi$ interpolating $\bu$ at these points:
    \begin{equation*}
        \bsvarphi(\bx) = \bu(\bx_0) + \bF(\bx-\bx_0), \quad \text{where $\bF = (\bF_{\cdot 1}| \bF_{\cdot 2}| \cdots| \bF_{\cdot n})$ and } \bF_{\cdot j} = \frac{4}{r}(\bu(\bx_j) - \bu(\bx_0)). 
    \end{equation*}
    Then $\bsvarphi(\bx_i) = \bu(\bx_i)$ for $i = 0,1,\ldots,n$.

    Now define $\bspsi(\bx) = \bu(\bx_0) + \bF_{skw}(\bx-\bx_0) = \bsvarphi(\bx) - \bF_{sym}(\bx-\bx_0)$, where $\bF_{skw} = \frac{\bF - \bF^T}{2}$ and $\bF_{sym} = \frac{\bF + \bF^T}{2}$. To prove the theorem it suffices to show that the function $\bv(\bx) := \bu(\bx) - \bspsi(\bx)$ satisfies
    \begin{equation}\label{eq:LInftyKorn:Pf1}
        \sup_{\bx \in \overline{G}} |\bv(\bx) \cdot \be_i| \leq C(n) M r^s, \qquad i = 1,\ldots,n.
    \end{equation}
    To this end, for $\bx\in \overline{G}$, $|\bx-\bx_0|\leq 2r$ and $|\bx -\bx_i |\leq 9r/4.$ We also have 
    \begin{equation*}
        \begin{split}
            \bv(\bx) \cdot \be_i = \frac{4}{r} \bv(\bx) \cdot (\bx_i - \bx_0) &=  \frac{4}{r} \big( \bv(\bx) \cdot (\bx - \bx_0) - \bv(\bx) \cdot (\bx - \bx_i) \big) \\
            &= \frac{4}{r} \big( (\bv(\bx) -\bv(\bx_0)) \cdot (\bx - \bx_0) - (\bv(\bx) - \bv(\bx_i)) \cdot (\bx - \bx_i) \big) \\
            &\qquad + \frac{4}{r} \big(  \bv(\bx_0) \cdot (\bx - \bx_0) - \bv(\bx_i) \cdot (\bx-\bx_i) \big).
        \end{split}
    \end{equation*}
    Since $\bspsi \in \cR$, $(\bv(\bx)-\bv(\by)) \cdot (\bx-\by) = (\bu(\bx) - \bu(\by)) \cdot (\bx-\by) $, and so the first term on the right-hand side is bounded by $4 M r^s$ by assumption, uniformly in $\bx$. To bound the second term, we use that $|\bx-\bx_i| \leq 9r/4$ for all $\bx \in \overline{G}$, $i = 0,1,\ldots,n$:
    \begin{equation*}
        \max_{i=1,\ldots,n} \left| \frac{4}{r} \big(  \bv(\bx_0) \cdot (\bx - \bx_0) - \bv(\bx_i) \cdot (\bx-\bx_i) \big) \right| \leq C(n) \max_{i=1,\ldots,n} |\bv(\bx_i)|.
    \end{equation*}
    Therefore in order to prove \eqref{eq:LInftyKorn:Pf1} we need to show that
    \begin{equation}\label{eq:LInftyKorn:Pf2}
        |\bv(\bx_i)| \leq C(n) M r^{s}, \qquad i = 0,1,\ldots,n.
    \end{equation}
    
    Now, we use the fact that $\bsvarphi$ was chosen via interpolation and that  $$\bv(\bx_i) = \bu(\bx_i) - \bspsi(\bx_i) = \bu(\bx_i) - \bsvarphi(\bx_i) + \bF_{sym}(\bx_i-\bx_0) = \bF_{sym}(\bx_i-\bx_0). $$  Since $|\bx_i - \bx_0| \leq r/4$ for all $i$, \eqref{eq:LInftyKorn:Pf2} will be established if we show that
    \begin{equation}\label{eq:LInftyKorn:Pf3}
        |\bF_{sym}| \leq C(n) M r^{s-1}.
    \end{equation}
    But we have that
    \begin{equation*}
        (\bF_{sym})_{ii} = \frac{16}{r^2} (\bF (\bx_i - \bx_0)) \cdot (\bx_i-\bx_0) = \frac{16}{r^2} (\bu(\bx_i) - \bu(\bx_0)) \cdot (\bx_i - \bx_0),
    \end{equation*}
    and so $|(\bF_{sym})_{ii}| \leq \frac{C}{r^2} M r^{1+s} = C(n) M r^{s-1}$. To control the off-diagonal entries, we use this estimate along with the identity
    \begin{equation*}
    \begin{split}
        (\bF_{sym})_{ij} &= \frac{16}{r^2} (\bF_{sym} (\bx_i - \bx_0)) \cdot (\bx_j-\bx_0) \\
        &= \frac{8}{r^2} 
        \big( (\bF (\bx_i - \bx_0)) \cdot (\bx_i-\bx_0)
        + (\bF (\bx_j - \bx_0)) \cdot (\bx_j-\bx_0)
        - (\bF (\bx_i - \bx_j)) \cdot (\bx_i-\bx_j) \big)
    \end{split}
    \end{equation*}
    to conclude that \eqref{eq:LInftyKorn:Pf3}, and hence the lemma, holds.
\end{proof}

\begin{theorem}\label{thm:Scaled:Infty}
    There exists a constant $C(n) > 0$ such that for all $0 < s < 1$ the following holds: for all $\bu \in \cK^s(B)$ there exists $\bspsi \in \cR$ such that
    \begin{equation*}
        \vnorm{\bu- \bspsi}_{L^\infty(B)} + (1-s)[\bu - \bspsi]_{C^s(B)} \leq C [\bu]_{\cK^s(B)}.
    \end{equation*}
\end{theorem}

\begin{proof}
Using Lemma \ref{lma:Unscaled:Infty}, let $\bspsi = \bA \bx + \bfb \in \cR$  such that $\vnorm{\bu - \bspsi}_{L^\infty(B)} \leq C(n) [\bu]_{\cK^s(B)}$, corresponding to $\bx_0 = 0$ and $r = 1$. 
For $\bx_0 \in \overline{B}$ and $r \leq 1$, define $B_{\bx_0,r} = B \cap B(\bx_0,r)$.
    Then
    \begin{equation*}
        B((1-r/2)\bx_0,r/2) \subset B_{\bx_0,r} \subset B((1-r/2)\bx_0,2r).
    \end{equation*}
    Thus we can apply Lemma \ref{lma:Unscaled:Infty} to obtain an affine vector field $\bspsi_{\bx_0,r} \in \cR$ satisfying
    \begin{equation*}
        \vnorm{\bu - \bspsi_{\bx_0,r}}_{L^\infty(B_{\bx_0,r})} \leq C(n) [\bu]_{\cK^s(B_{\bx_0,r})} r^s.
    \end{equation*}
    define $B_k = B_{\bx_0,2^{-k}}$, and define $\bspsi_k$ via $\bspsi_0 = \bspsi$ and $\bspsi_{k} = \bspsi_{\bx_0,2^{-k}}$ for $k \geq 1$. We write $\bspsi_k(\bx) = \bA_k \bx + \bfb_k$, where $\bA_k \in \bbR^{n \times n}$ and $\bfb_k \in \bbR^n$.
    Then by the triangle inequality
    \begin{equation*}
        \vnorm{\bspsi_{k} - \bspsi_{k+1}}_{L^\infty(B_{k+1})} \leq \vnorm{\bu - \bspsi_{k+1}}_{L^\infty(B_{k+1})} + \vnorm{\bu - \bspsi_{k}}_{L^\infty(B_k)}\leq C(n) [\bu]_{\cK^s(B)} 2^{-ks},
    \end{equation*}
    for $k \in \bbN_0$.
    Moreover, \eqref{eq:AffineMapEquivalence:Rescaled:MatrixPart} implies that
    \begin{equation*}
        |\bA_k - \bA_{k+1}| \leq 2^{k} C \vnorm{\bspsi_{k} - \bspsi_{k+1}}_{L^\infty(B_{k+1})} \leq C [\bu]_{\cK^s(B)} 2^{-k(s-1)},
    \end{equation*}
    for $k \in \bbN_0$. Therefore, by the triangle inequality, for any $k \in \bbN$ 
    \begin{equation*}
        |\bA - \bA_k| \leq C(n) [\bu]_{\cK^s(B)} \sum_{j=0}^k 2^{-j(s-1)} \leq \frac{C(n) }{1 - 2^{s-1}} [\bu]_{\cK^s(B)} 2^{-k(s-1)}.
    \end{equation*}
    
 We are now ready to estimate the H\"older seminorm. Let $\bx$, $\by \in B$ with $|\bx-\by| := h$. Then $h \in (0,2)$. If $h \geq 1$ then $$|\bu(\bx)-\bu(\by)-\bA(\bx-\by)| \leq 2 \vnorm{\bu - \bspsi}_{L^\infty(B)} \leq C [\bu]_{\cK^s(B)},$$ so we may assume that $h < 1$. Then there exists $k \in \bbN$ such that $2^{-k-1} < h \leq 2^{-k}$. 
 
 Now set $\bx_0 = \frac{\bx+\by}{2}$. Applying the above procedure for this $\bx_0,$ it follows that there exists $\bspsi_k$ such that, since $\bx$, $\by \in B_{\bx_0,2^{-k+1}}$,  
    \begin{equation*}
    \begin{split}
        |\bu(\bx) - \bu(\by) - \bA(\bx-\by)| &\leq |\bu(\bx) - \bspsi_k(\bx)| + |\bu(\by) - \bspsi_k(\by)| + |\bA_k - \bA| h \\
        &\leq C(n) [\bu]_{\cK^s(B)} \left( 2^{-ks} + \frac{2^{-k(s-1)}}{ 1 - 2^{s-1} }h \right)\\
        &\leq \frac{C(n)}{1-s} [\bu]_{\cK^s(B)}  h^s.
    \end{split}
    \end{equation*}
    So $\frac{|\bu(\bx) - \bu(\by) - \bA(\bx-\by)|}{|\bx-\by|^s} \leq \frac{C(n)}{1-s} [\bu]_{\cK^s(B)}$, as desired.
\end{proof}

\begin{theorem}\label{thm:Korns1stOnBall:Linfty}
    There exists $C = C(n) > 0$ such that for all $\bu \in \cK^{s}(B)$
    \begin{equation}
        \vnorm{\bu - \Pi_B[\bu]}_{L^\infty(B)} + (1-s)[\bu - \Pi_B[\bu]]_{C^{s}(B)} \leq C [\bu]_{\cK^{s}(B)}.
    \end{equation}
\end{theorem}

\begin{proof}
    Let $\bspsi \in \cR$ be as in the previous theorem. Then by \eqref{eq:ProjContLp}
    \begin{equation*}
    \begin{split}
        \vnorm{\bu - \Pi_B[\bu]}_{L^\infty(B)} &\leq \vnorm{\bu - \bspsi}_{L^\infty(B)} + \vnorm{\Pi_B[\bu] - \bspsi}_{L^\infty(B)}
        \leq (1 + C(n))\vnorm{\bu - \bspsi}_{L^\infty(B)},
    \end{split}
    \end{equation*}
    and by \eqref{eq:AffineMapEquivalence:Rescaled}
    \begin{equation*}
    \begin{split}
        [\bu - \Pi_B[\bu]]_{C^s(B)} &\leq [\bu - \bspsi]_{C^s(B)} + [\Pi_B[\bu]-\bspsi]_{C^s(B)} \\
        &= [\bu - \bspsi]_{C^s(B)} + [\Pi_B[\bu-\bspsi]]_{C^s(B)} \\
        &\leq C(n) \left( [\bu - \bspsi]_{C^s(B)} + \vnorm{\bu - \bspsi}_{L^\infty(B)} \right).
    \end{split}
    \end{equation*}
    Therefore,
    \begin{equation*}
        \vnorm{\bu - \Pi_B[\bu]}_{L^\infty(B)} + (1-s)[\bu - \Pi_B[\bu]]_{C^{s}(B)} \leq C(n,p) \left( \vnorm{\bu - \bspsi}_{L^\infty(B)} + (1-s)[\bu - \bspsi]_{C^{s}(B)}  \right),
    \end{equation*}
    and the result follows from the previous lemma.
\end{proof}
And, finally we present a rescaled estimate as 
\begin{theorem}\label{thm:KornBallStrong:Linfty}
    Let $\bx_0 \in \bbR^n$, $R > 0$ and $B_R = B(\bx_0,R)$. There exists $C = C(n) > 0$ such that for all $\bu \in \cK^s(B)$    
    \begin{equation}
        R^{-s} \vnorm{\bu - \Pi_{B_R}[\bu]}_{L^\infty(B_{R})} + (1-s)[\bu - \Pi_{B_{R}}[\bu]]_{C^s(B_{R})} \leq C [\bu]_{\cK^{s}(B_{R})}.
    \end{equation}
\end{theorem}

\section{Fractional Korn's Inequality on Bounded Lipschitz Domains}\label{sec:lipschitz}

To extend the fractional Korn inequality from balls to general bounded Lipschitz domains, we first examine bi-Lipschitz images of Euclidean balls. The following is a type of Whitney covering of the unit ball using Euclidean balls instead of cubes; see \cite{Kohn1982}.

\begin{proposition}\label{prop:WhitneyBallCovering}
    For a constant $M_0 \geq 1$ and $j \in \mathbb{N}_0$, set
    \begin{equation}
        \mu := \frac{1}{4M_0}, \quad R_j := 1 - (1-\mu)^j, \quad r_j := 2(R_{j+1} - R_j).
    \end{equation}
    Then there exist collections $\mathcal{F}_j = \{ B(\bx, r_j) : \bx \in S_j \}$ centered on concentric spheres $S_j \subset \partial B(0, R_j)$ such that:
    \begin{enumerate}
        \item[\textbf{(i)}] $\cF_0 = \{ B(0,r_0)\}$, $\cF_j$ is an open cover of $B(0,\frac{R_{j+1}+R_j}{2}) \setminus B(0,\frac{R_{j}+R_{j-1}}{2})$, and $\bigcup_{j=0}^\infty \mathcal{F}_j$ is an open cover of $B(0,1)$. 
        \item[\textbf{(ii)}] We say that $B \in \cF_j$ and $B' \in \cF_{j+1}$ are \textit{connected} if their intersection contains a ball of radius $\frac{r_{j+1}}{2}$. For every $B \in \cF_j$ there exists $B' \in \cF_{j+1}$ such that $B$ and $B'$ are connected.
        \item[\textbf{(iii)}] For $j < k$, a sequence $\{B_j,B_{j+1},\ldots,B_k\}$ is a \textit{chain} from $B \in \cF_j$ to $B' \in \cF_k$ if $B_j = B$, $B_k = B'$, $B_\ell \in \cF_\ell$ for $\ell = j, j+1, \ldots, k$, and $B_\ell$ is connected to $B_{\ell + 1}$ for $\ell = j, \ldots, k-1$. The number of chains between $\mathcal{F}_j$ and $\mathcal{F}_k$ is at most $N_1 M_0^{n-1} (r_j/r_k)^{n-1}$, where $N_1 = N_1(n) > 0$.
        \item[\textbf{(iv)}] The dilated collection $\{ M_0 B : B \in \bigcup_{j=0}^\infty \mathcal{F}_j \}$ remain a locally finite cover; any $\bx \in B(0,1)$ is contained in at most $N_2 M_0^n$ of the sets in said collection, where $N_2 = N_2(n) > 0$.
    \end{enumerate}
\end{proposition}

We also record a matrix norm comparison lemma on balls:
\begin{lemma}\label{lem:MatrixNormComp}
    Let $B_1 = B(\bx_1,R_1)$, $B_2 = B(\bx_2,R_2)$,  $B_3 = B(\bx_3,R_3)$ be Euclidean balls in $\mathbb{R}^n$ with $B_1, B_2 \subset B_3$. Then for any matrix $\mathbf{A} \in \mathbb{R}^{n \times n}$, $s \in (0,1)$, and $p \in [1,\infty)$,
    \begin{equation}\label{eq:MatrixNormComp}
        \left( \int_{B_1} \fint_{B_1} \frac{|\mathbf{A}\bx - \mathbf{A}\by|^p}{|\bx-\by|^{n+sp}} \, \mathrm{d}\by\mathrm{d}\bx \right)^{1/p} \leq C(n,p) 
        \left( \frac{R_3}{R_2} \right)^{1-s}
        \left( \int_{B_2} \fint_{B_2} \frac{|\mathbf{A}\bx - \mathbf{A}\by|^p}{|\bx-\by|^{n+sp}} \, \mathrm{d}\by\mathrm{d}\bx \right)^{1/p}.
    \end{equation}
\end{lemma}
\begin{proof}
    For any ball $B(\bx_0, R)$, substituting $\bx = \bx_0 + R\mathbf{z}$ and $\by = \bx_0 + R\mathbf{w}$ gives the scaling relation:
    \begin{equation}\label{eq:BallSeminormScaling}
        \left( \int_{B(\bx_0, R)} \fint_{B(\bx_0, R)} \frac{|\mathbf{A}\bx - \mathbf{A}\by|^p}{|\bx-\by|^{n+sp}} \, \mathrm{d}\by \, \mathrm{d}\bx \right)^{1/p} = R^{1-s} \left( \int_{B(0,1)} \fint_{B(0,1)} \frac{|\mathbf{A}\mathbf{z} - \mathbf{A}\mathbf{w}|^p}{|\mathbf{z}-\mathbf{w}|^{n+sp}} \, \mathrm{d}\mathbf{w} \, \mathrm{d}\mathbf{z} \right)^{1/p}.
    \end{equation}
    Applying \eqref{eq:NormCompEquiv} and \eqref{eq:BallSeminormScaling} to $B_1 = B(\bx_1, R_1)$ and $B_2 = B(\bx_2, R_2)$,
    \begin{equation}
        \begin{aligned}
            \left( \int_{B_1} \fint_{B_1} \frac{|\mathbf{A}\bx - \mathbf{A}\by|^p}{|\bx-\by|^{n+sp}} \, \mathrm{d}\by \, \mathrm{d}\bx \right)^{1/p} 
            &\leq \frac{C_2}{1-s} R_1^{1-s} |\mathbf{A}| \\
            &\leq \frac{C_2}{C_1} \left( \frac{R_1}{R_2} \right)^{1-s} \left( \int_{B_2} \fint_{B_2} \frac{|\mathbf{A}\bx - \mathbf{A}\by|^p}{|\bx-\by|^{n+sp}} \, \mathrm{d}\by \, \mathrm{d}\bx \right)^{1/p}, 
        \end{aligned}
    \end{equation}
    which proves \eqref{eq:MatrixNormComp} after noting that $R_1 \le R_3$.
\end{proof} 

\begin{lemma}[\cite{Kohn1982}, Lemma~4.11]\label{lem:AffineNormComp}
    Let $B_1 = B(\bx_1,R_1)$, $B_2 = B(\bx_2,R_2)$,  $B_3 = B(\bx_3,R_3)$ be Euclidean balls in $\mathbb{R}^n$ with $B_1, B_2 \subset B_3$. Then for any affine map $\bspsi(\bx) = \mathbf{A}\bx + \bfb$, and $p \in [1,\infty)$,
    \begin{equation}\label{eq:AffineNormComp}
        \vnorm{\bspsi}_{L^p(B_1)} \leq C(n,p) \left( \frac{R_1}{R_2} \right)^{n/p}
        \frac{R_3}{R_2}
        \vnorm{\bspsi}_{L^p(B_2)}.
    \end{equation}
\end{lemma}

\begin{theorem}[Korn Inequality on a bi-Lipschitz image of a ball]\label{thm:KornBiLipBall}
    Let $p \in [1,\infty)$ and $0 < s < 1$.
    Let $\Omega = \bsgamma(B(0,1))$, where $\bsgamma : B(0,1) \to \Omega$ is a bi-Lipschitz mapping with Lipschitz constant $L \geq 1$. Then for every $\bu \in \mathcal{X}^{s,p}(\Omega)$, there exists $\bspsi \in \mathcal{R}$ with $\vnorm{\bspsi}_{\mathrm{aff}} \leq C(n,p,\bsgamma,L) \vnorm{\bu}_{L^p(\Omega)}$ such that 
    \begin{equation}
        (1-s)^{-1} \vnorm{\bu - \bspsi}_{L^p(\Omega)}^p + [\bu - \mathbf{\bspsi}]_{W^{s,p}(\Omega)}^p \leq C \alpha_s(p) [\bu]_{\mathcal{X}^{s,p}(\Omega)}^p,
    \end{equation}
    where $C = C(n,p,L) > 0$ and $\alpha_s(p)$ is defined in Theorem \ref{thm:mainKorn2nd}.
\end{theorem}

\begin{proof}
    Let $\bsgamma : B(0,1) \to \Omega$ be a bi-Lipschitz mapping with constant $L \geq 1$, satisfying
    \begin{equation}
        \frac{1}{L} |\bx-\by| \leq |\bsgamma(\bx) - \bsgamma(\by)| \leq L |\bx-\by| \quad \text{for all } \bx, \by \in B(0,1).
    \end{equation}
    It follows that
    \begin{equation}\label{eq:ComparabilityOfCovers}
        \bsgamma\big(B(\bx, r/L)\big) \subset B\big(\bsgamma(\bx), r\big) \subset \bsgamma\big(B(\bx, Lr)\big)
    \end{equation}
    for any $\bx \in B(0,1)$ and $r \in \big(0, \frac{1}{L}(1-|\bx|)\big)$.
    
    We construct an open covering of $\Omega$ as follows: set $M_0 := L^2$, let $\mu, R_j, r_j$ be as in Proposition \ref{prop:WhitneyBallCovering}, and let $\mathcal{F}_j = \{ B(\bx, r_j) : \bx \in S_j \}$ be the Whitney cover of $B(0,1)$. Define the collections
    \begin{equation}
        \mathcal{G}_j := \big\{ B(\bsgamma(\bx), L r_j) : \bx \in S_j \big\} \quad \text{and} \quad 2\mathcal{G}_j := \big\{ B(\bsgamma(\bx), 2 L r_j) : \bx \in S_j \big\}.
    \end{equation}
    By Proposition \ref{prop:WhitneyBallCovering} and \eqref{eq:ComparabilityOfCovers}, these collections satisfy the following:
    \begin{enumerate}
        \item[\textbf{(i)}] For each $j \in \mathbb{N}_0$, all elements of $2\mathcal{G}_j$ are contained in $\Omega$.
        \item[\textbf{(ii)}] For each $j \in \mathbb{N}$, the elements of $\mathcal{G}_j$ cover the set $\bsgamma\big( B(0,\frac{R_{j+1}+R_j}{2}) \setminus B(0,\frac{R_j+R_{j-1}}{2}) \big)$, with $\mathcal{G}_0 = \{ B(\bsgamma(0), L r_0) \}$. Thus $\bigcup_{j=0}^{\infty} \mathcal{G}_j$ forms an open cover of $\Omega$. 
        \item[\textbf{(iii)}] Two balls $B \in \mathcal{G}_j$ and $B' \in \mathcal{G}_{j+1}$ are \emph{connected} if their intersection $B \cap B'$ contains a ball of radius $\frac{r_{j+1}}{2L}$. For every $B \in \mathcal{G}_j$ there exists $B' \in \cG_{j+1}$ such that $B$ and $B'$ are connected.
        \item[\textbf{(iv)}] For $j < k$, a sequence $\{B_j, B_{j+1}, \dots, B_k\}$ is a \emph{chain} from $B \in \mathcal{G}_j$ to $B' \in \mathcal{G}_k$ if $B_j = B$, $B_k = B'$, $B_\ell \in \mathcal{G}_\ell$ for $\ell = j,j+1,\ldots,k$, and $B_\ell$ is connected to $B_{\ell+1}$ for all $\ell = j, \dots, k-1$. The following properties hold:
        \begin{enumerate}
            \item[\textbf{(a)}] If $B = B(\bx,r_j) \in \mathcal{F}_j$ and $B' = B(\by,r_{j+1}) \in \mathcal{F}_{j+1}$ intersect as in Proposition \ref{prop:WhitneyBallCovering}, then $B(\bsgamma(\bx), L r_j ) \in \mathcal{G}_j$ and $B(\bsgamma(\by), L r_{j+1}) \in \mathcal{G}_{j+1}$ are connected in $\Omega$, and the ball contained in their intersection can be centered at $\bsgamma(\mathbf{z})$ for some $\mathbf{z} \in B \cap B'$.
            \item[\textbf{(b)}] If for $j < k$ there exists a chain from $B(\bx,r_j) \in \cF_j$ to $B(\by,r_k) \in \cF_k$ as in the definition of Proposition \ref{prop:WhitneyBallCovering}, then there exists a chain from $B(\bsgamma(\bx),Lr_j) \in \cG_j$ to $B(\bsgamma(\by),Lr_k) \in \cG_k$.
            \item[\textbf{(c)}] For $j < k$, the total number of chains from a fixed element of $\cG_j$ to any element in $\cG_k$ is at most $N_1 M_0^{n-1} (r_j / r_k)^{n-1}$, where $N_1 = N_1(n) > 0$.
        \end{enumerate}
        \item[\textbf{(iv)}] Any point in $\Omega$ is contained in at most $N_2 M_0^n$ elements of $\bigcup_{j=0}^\infty 2\mathcal{G}_j$, where $N_2 = N_2(n) > 0$.
    \end{enumerate}
    
    Set $B_0 := B(\bsgamma(0), L r_0) \in \mathcal{G}_0$. 
    For brevity we define for an open set $G \subseteq \Omega$
    \begin{equation*}
        E^{s,p}_G(\bv) := \left( (1-s)^{-1} \vnorm{\bv}_{L^p(G)}^p + [\bv]_{W^{s,p}(G)}^p \right)^{1/p}
    \end{equation*}
    which is a norm on $W^{s,p}(G;\bbR^n)$.
    By Theorem \ref{thm:KornBallStrong} (Theorem \ref{thm:KornBallStrong:L1} in the case $p = 1$), the rigid motion $\bspsi_0 = \Pi_{B_0}[\bu] \in \mathcal{R}$ satisfies
    \begin{equation}
        \big( E^{s,p}_{B_0}(\bu-\bspsi_0) \big)^p \leq C(n,p) \alpha_s(p) [\bu]_{\mathcal{X}^{s,p}(B_0)}^p, 
    \end{equation}
    where we have used the fact that $(2L)^s \geq 1.$
    To estimate $\bu - \bspsi_0$ on an arbitrary ball $B' \in \mathcal{G}_k$ ($k \in \mathbb{N}_0$), let $\{ B_0, B_1, \dots, B_k = B' \}$ be a chain from $B_0$ to $B'$. On each ball $B_j \in \mathcal{G}_j$ ($j = 1, \dots, k$), let $\bspsi_j = \Pi_{B_j}[\bu] \in \mathcal{R}$ satisfy
    \begin{equation}
        \big( E^{s,p}_{B_j}(\bu-\bspsi_j) \big)^p  \leq C(n,p) \alpha_s(p) [\bu]_{\mathcal{X}^{s,p}(B_j)}^p,
    \end{equation}
    again thanks to Theorem \ref{thm:KornBallStrong} (or Theorem \ref{thm:KornBallStrong:L1}).
    By the triangle inequality applied to $E^{s,p}$
    \begin{equation}\label{eq:NonlinPoinProof:Estimate1}
        E^{s,p}_{B_k}(\bu - \bspsi_0) \leq E^{s,p}_{B_k}(\bu - \bspsi_k) + \sum_{j=0}^{k-1} E^{s,p}_{B_k}(\bspsi_{j+1} - \bspsi_j).
    \end{equation}
    Since $B_j \cap B_{j+1}$ contains a ball $B(\bsgamma(\mathbf{z}_j), \frac{r_{j+1}}{2L})$ centered at $\bsgamma(\mathbf{z}_j)$, Theorem \ref{thm:KornBallStrong} (or Theorem \ref{thm:KornBallStrong:L1}) gives
    \begin{equation}\label{eq:NonlinPoinProof:Estimate2}
        \begin{aligned}
            E^{s,p}_{B(\bsgamma(\bz_j),\frac{r_{j+1}}{2L})}(\bspsi_{j+1} - \bspsi_j)
            &\leq E^{s,p}_{B_j}(\bu - \bspsi_j) + E^{s,p}_{B_{j+1}}(\bu - \bspsi_{j+1})
            \\
            &\leq C(n,p) \alpha_s(p)^{1/p} \left( [\bu]_{\mathcal{X}^{s,p}(B_j)} + [\bu]_{\mathcal{X}^{s,p}(B_{j+1})} \right).
        \end{aligned}
    \end{equation}
    Next, we apply Lemma \ref{lem:MatrixNormComp} to modify the domain of integration for the seminorm of the affine maps. Since $B_k = B(\bsgamma(\bx_k), L r_k)$ and $B(\bsgamma(\mathbf{z}_j), \frac{r_{j+1}}{2L})$ are both contained in a ball centered at $\frac{\bsgamma(\bx_k) + \bsgamma(\mathbf{z}_j)}{2}$ with radius $Lr_k + \frac{r_{j+1}}{2L} + |\bsgamma(\mathbf{z}_j) - \bsgamma(\bx_k)|$, and since
    \begin{equation}
        |\bsgamma(\mathbf{z}_j) - \bsgamma(\bx_k)| \leq L \max\{ |\bx_j - \bx_k|, |\bx_{j+1} - \bx_k| \} \leq 2L (1 - R_j) = 4 L M_0 r_j,
    \end{equation}
    we have by Lemma \ref{lem:MatrixNormComp}, Lemma \ref{lem:AffineNormComp} and \eqref{eq:NonlinPoinProof:Estimate2}
    \begin{equation}\label{eq:NonlinPoinProof:Estimate3} 
        E^{s,p}_{B_k}(\bspsi_{j+1}-\bspsi_j) \leq C(n,p,L) \alpha_s(p)^{1/p} \left( \left( \frac{r_k}{r_j} \right)^{\frac{n}{p}} [\bu]_{\mathcal{X}^{s,p}(B_j)} + \left( \frac{r_k}{r_{j+1}} \right)^{\frac{n}{p}} [\bu]_{\mathcal{X}^{s,p}(B_{j+1})} \right),
    \end{equation}
    where we used that $(1+\frac{2L^2 r_k}{r_j} + 4 L^2 M_0)^{1-s} \leq C(n,L)$ for all $s \in [0,1]$.
    Inserting \eqref{eq:NonlinPoinProof:Estimate3} into \eqref{eq:NonlinPoinProof:Estimate1},
    \begin{equation} 
        E^{s,p}_{B_k}(\bu-\bspsi_0)
        \leq C(n,p,L) \alpha_s(p)^{1/p} \sum_{j=0}^k \left( \frac{r_k}{r_j} \right)^{\frac{n}{p}} [\bu]_{\mathcal{X}^{s,p}(B_j)}.
    \end{equation}
    For any $\beta \in (0,1]$, Hölder's inequality for $p > 1$ and the geometric series formula gives
    \begin{equation}
    \begin{aligned}
        \sum_{j=0}^k r_j^{-\frac{n}{p}} [\bu]_{\mathcal{X}^{s,p}(B_j)} &\leq \left( \sum_{j=0}^k r_j^{-\frac{\beta}{p-1}} \right)^{\frac{1}{p'}} \left( \sum_{j=0}^k r_j^{\beta-n} [\bu]_{\mathcal{X}^{s,p}(B_j)}^p \right)^{\frac{1}{p}}\\
        &\leq C(p,M_0) r_k^{-\frac{\beta}{p}} \left( \sum_{j=0}^k r_j^{\beta-n} [\bu]_{\mathcal{X}^{s,p}(B_j)}^p \right)^{\frac{1}{p}}.
    \end{aligned}
    \end{equation}
    (The estimate when $p = 1$ uses only the inequality $r_k \leq r_j$.)
    Therefore, writing $B_k = B'$
    \begin{equation}
        \big( E^{s,p}_{B'}(\bu-\bspsi_0) \big)^p
        \leq C(n,p,L) \alpha_s(p) \sum_{j=0}^k \left( \frac{r_j}{r_k} \right)^{\beta-n} [\bu]_{\mathcal{X}^{s,p}(B_j)}^p.
    \end{equation}
    Now set $\eta := \frac{1}{3L^2} < 1$. For any $\bx \in B' = B(\bsgamma(\bx_k), Lr_k)$, since $\mathrm{dist}(\bx,\partial\Omega) \leq 2 L r_k$, any point $\by$ satisfying $|\by-\bx| \leq \eta \mathrm{dist}(\bx,\partial\Omega)$ satisfies
    \begin{equation}
        |\by - \bsgamma(\bx_k)| \leq |\by-\bx| + |\bx - \bsgamma(\bx_k)| < \frac{2Lr_k}{3L^2} + Lr_k \leq 2 L r_k.
    \end{equation}
    Thus $\{ \by : |\by-\bx| \leq \eta \mathrm{dist}(\bx,\partial\Omega) \} \subset 2B'$. Therefore, we can estimate
    \begin{equation}
        \begin{aligned}
            \int_{B'} \int_{\{ |\by-\bx| \leq \eta \mathrm{dist}(\bx,\partial\Omega) \}} &\frac{|(\bu-\bspsi_0)(\bx) - (\bu-\bspsi_0)(\by)|^p}{|\bx-\by|^{n+sp}} \, \mathrm{d}\by \, \mathrm{d}\bx \\
            &\leq [\bu-\bspsi_0]_{W^{s,p}(2B')}^p \leq C(n,p,L) \alpha_s(p) \sum_{j=0}^k \left( \frac{r_j}{r_k} \right)^{\beta-n} [\bu]_{\mathcal{X}^{s,p}(2B_j)}^p.
        \end{aligned}
    \end{equation}
    Summing over all $B' \in \mathcal{G}_k$ and applying the bound (iv)(c) on the number of chains, we obtain for each $k \in \mathbb{N}_0$:
    \begin{equation}
        \begin{aligned}
            \sum_{B' \in \mathcal{G}_k} \Bigg( \frac{\vnorm{\bu-\bspsi_0}_{L^p(B')}^p}{1-s} + \int_{B'} & \int_{\{ |\by-\bx| \leq \eta \mathrm{dist}(\bx,\partial\Omega) \}}  \frac{|(\bu-\bspsi_0)(\bx) - (\bu-\bspsi_0)(\by)|^p}{|\bx-\by|^{n+sp}} \, \mathrm{d}\by \, \mathrm{d}\bx \Bigg) \\
            &\leq C(n,p,L) \alpha_s(p) \sum_{j=0}^k \sum_{B \in 2\mathcal{G}_j} \left( \frac{r_j}{r_k} \right)^{\beta-1} [\bu]_{\mathcal{X}^{s,p}(B)}^p.
        \end{aligned}
    \end{equation}
    Summing over $k \in \mathbb{N}_0$ and applying Fubini's theorem
    \begin{equation}
        \begin{aligned}
            \frac{\vnorm{\bu-\bspsi_0}_{L^p(\Omega)}^p}{1-s} + \int_{\Omega} & \int_{\{ |\by-\bx| \leq \eta \mathrm{dist}(\bx,\partial\Omega) \}} \frac{|(\bu-\bspsi_0)(\bx) - (\bu-\bspsi_0)(\by)|^p}{|\bx-\by|^{n+sp}} \, \mathrm{d}\by \, \mathrm{d}\bx \\
            &\leq C(n,p,L) \alpha_s(p)  \sum_{k=0}^{\infty} \sum_{j=0}^k \sum_{B \in 2\mathcal{G}_j} \left( \frac{r_j}{r_k} \right)^{\beta-1} [\bu]_{\mathcal{X}^{s,p}(B)}^p \\
            &= C(n,p,L) \alpha_s(p)  \sum_{j=0}^{\infty} r_j^{\beta-1} \left( \sum_{k=j}^{\infty} r_k^{1-\beta} \right) \sum_{B \in 2\mathcal{G}_j} [\bu]_{\mathcal{X}^{s,p}(B)}^p \\
            &\leq C(n,p,L) \alpha_s(p)  \sum_{j=0}^{\infty} \sum_{B \in 2\mathcal{G}_j} [\bu]_{\mathcal{X}^{s,p}(B)}^p \leq C(n,p,L) \alpha_s(p) [\bu]_{\mathcal{X}^{s,p}(\Omega)}^p,
        \end{aligned}
    \end{equation}
    where the final inequality follows from the local finiteness of $\bigcup_{j=0}^\infty 2\mathcal{G}_j$. 

    It is shown in \cite{Dyda} (see also \cite{HurriVahakangas,PratsSaksman}) that, when $\Omega$ is Lipschitz, the double integral on the left-hand side is equivalent to the $W^{s,p}$ seminorm $[\bu-\bspsi_0]_{W^{s,p}(\Omega)}^p$, with constant of comparison depending only on $n$, $p$, $L$, and $\eta$. We therefore conclude
    \begin{equation}
        (1-s)^{-1} \vnorm{\bu-\bspsi_0}_{L^p(\Omega)}^p + [\bu - \bspsi_0]_{W^{s,p}(\Omega)}^p \leq C(n,p,L) \alpha_s(p) [\bu]_{\mathcal{X}^{s,p}(\Omega)}^p.
    \end{equation}
    Last, we note that $\vnorm{\bspsi_0}_{\mathrm{aff}} \leq C(n,p,\bsgamma,L) \vnorm{\bu}_{L^p(B_0)}$ by \eqref{eq:AffineMapEquivalence:Rescaled} since $B_0(\bsgamma(0),L\mu)$, and $\mu$ depends only on $M_0$.
    This completes the proof.
\end{proof}

\begin{proof}[Proof of Theorem \ref{thm:mainKorn1st} and of Theorem \ref{thm:mainKorn1st:Endpoints} when $p = 1$]
	Cover $\partial \Omega$ by finitely many sets of the form (up to translation and rotation)
	\begin{equation*}
		\{ \bx = (\bx',x^n) : 0 < x^n < \beta(\bx'), \, |\bx'| < r \}, \quad r > 0
	\end{equation*}
    where $\beta$ is a Lipschitz function and where  
	each of these sets is a bi-Lipschitz image of $B(0,1)$. Combining these with open balls covering the interior of $\Omega$, we obtain a finite open cover $\{ \Omega_j \}_{j=1}^K$ of $\Omega$, $K$ depending on $\Omega$, such that each $\Omega_j$ is a bi-Lipschitz image of $B(0,1)$ with constant $L_j \ge 1$. Moreover, we enumerate $\{\Omega_j\}$ so that $\Omega_j \cap \Omega_{j+1} \neq \emptyset$ for $j = 1, \dots, K-1$.

	By Theorem \ref{thm:KornBiLipBall}, for each $j \in \{1, \dots, K\}$, there exists a rigid motion $\bspsi_j \in \mathcal{R}$ such that
	\begin{equation}\label{eq:LocalKornBiLip}  
    \begin{split}
        (1-s)\vnorm{\bu-\bspsi_j}_{L^p(\Omega_j)}^p + [\bu - \bspsi_j]_{W^{s,p}(\Omega_j)}^p &\leq C(n,p,L_j) \alpha_s(p) [\bu]_{\mathcal{X}^{s,p}(\Omega_j)}^p \\
        &\leq C(n,p,\Omega) \alpha_s(p) [\bu]_{\mathcal{X}^{s,p}(\Omega)}^p,
    \end{split}
	\end{equation}
	and $\|\bspsi_j\|_{\mathrm{aff}} \leq C(n,p,\Omega) \|\bu\|_{L^p(\Omega_j)} \le C \|\bu\|_{L^p(\Omega)}$.
    Thus
    \begin{equation*}
        \begin{split}
            [\bu-\bspsi_1]_{W^{s,p}(\Omega)} &\leq C(K) \sum_{j=1}^K [\bu-\bspsi_1]_{W^{s,p}(\Omega_j)} \\
            &\leq C(K) \sum_{j=1}^K \left( [\bu-\bspsi_j]_{W^{s,p}(\Omega_j)} + [\bspsi_{j}-\bspsi_1]_{W^{s,p}(\Omega_j)} \right) \\
            &\leq C(n,p,\Omega,K) \sum_{j=1}^K \left( [\bu-\bspsi_j]_{W^{s,p}(\Omega_j)} + \frac{1}{(1-s)^{1/p}} |\grad (\bspsi_{j}-\bspsi_1)| \right),
        \end{split}
    \end{equation*}
    where in the last inequality we used an estimation similar to that of the proof of \eqref{eq:NormCompEquiv}.
    Now, by \eqref{eq:NormCompEquiv} and Lemma \ref{lem:MatrixNormComp} applied on a ball $B_k \subset \Omega_k \cap \Omega_{k+1}$,
	\begin{equation}
    \begin{split}
        |\nabla \bspsi_{k+1} - \nabla \bspsi_k| &\leq C(n,p,\Omega) (1-s)^{1/p} [\bspsi_{k+1}-\bspsi_k]_{W^{s,p}(\Omega_k \cap \Omega_{k+1})} \\
        &\leq C(n,p,\Omega) (1-s)^{1/p} \left( [\bu - \bspsi_{k+1}]_{W^{s,p}(\Omega_{k+1})} + [\bu - \bspsi_{k}]_{W^{s,p}(\Omega_{k})} \right) \\
        &\leq C(n,p,\Omega)(1-s)^{1/p} \alpha_s(p)^{1/p} [\bu]_{\mathcal{X}^{s,p}(\Omega)} \quad \text{for any } k \in \{ 1, \dots, K-1 \}.
    \end{split}
	\end{equation}
    Therefore,
    \begin{equation*}
        |\grad(\bspsi_j - \bspsi_1)| \leq \sum_{k=1}^{j-1} |\grad(\bspsi_{k+1} - \bspsi_k)| \leq C(n,p,\Omega,K)(1-s)^{1/p} \alpha_s(p)^{1/p} [\bu]_{\mathcal{X}^{s,p}(\Omega)}.
    \end{equation*}
    In conclusion,
	\begin{equation*}
        \begin{split}
	    [\bu - \bspsi_1]_{W^{s,p}(\Omega)} 
            &\leq C(n,p,\Omega,K) \sum_{j=1}^K \left( [\bu-\bspsi_j]_{W^{s,p}(\Omega_j)} + \alpha_s(p)^{1/p} [\bu]_{\cX^{s,p}(\Omega)} \right) \\
            &\leq C(n,p,\Omega) \alpha_s(p)^{1/p} [\bu]_{\cX^{s,p}(\Omega)}.
        \end{split}
	\end{equation*}
	Thus, setting $\bspsi = \bspsi_1$, we  
    complete the proof for the seminorm.

    The bound for $(1-s)^{-1} \vnorm{\bu - \bspsi_1}_{L^p(\Omega)}$ follows in a similar way; since $\Omega \subset B(0,R_\Omega)$ for some $R_\Omega > 0$ depending only on $\Omega$,
    \begin{equation*}
        \begin{split}
            \vnorm{\bu-\bspsi_1}_{L^p(\Omega)} &\leq C(K) \sum_{j=1}^K \vnorm{\bu-\bspsi_1}_{L^p(\Omega_j)}  \\
            &\leq C(K) \sum_{j=1}^K \left( \vnorm{\bu-\bspsi_j}_{L^p(\Omega_j)}  + \vnorm{\bspsi_j-\bspsi_1}_{L^p(B(0,R_\Omega))}  \right).
        \end{split}
    \end{equation*}
    Now, by Lemma \ref{lem:AffineNormComp} applied on a ball $B_k \subset \Omega_k \cap \Omega_{k+1}$
    \begin{equation*}
    \begin{split}
        \vnorm{\bspsi_{k+1} - \bspsi_k}_{L^p(B(0,R_\Omega))} &\leq C(n,p,\Omega) \vnorm{\bspsi_{k+1} - \bspsi_k}_{L^p(\Omega_k \cap \Omega_{k+1})} \\
        &\leq C(n,p,\Omega) \big( \vnorm{\bu - \bspsi_{k+1}}_{L^p(\Omega_{k+1})} + \vnorm{\bu - \bspsi_{k}}_{L^p(\Omega_{k})} \big) \\
        &\leq C(n,p,\Omega) (1-s)^{1/p} \alpha_s(p)^{1/p} [\bu]_{\cX^{s,p}(\Omega)} \quad \text{ for any } k \in \{ 1,\ldots,K-1\}.
    \end{split}
    \end{equation*}
    Therefore,
    \begin{equation*}
        \vnorm{\bspsi_j-\bspsi_1}_{L^p(B(0,R_\Omega))} \leq \sum_{j=1}^{k-1} \vnorm{\bspsi_{k+1} - \bspsi_k}_{L^p(B(0,R_\Omega))} \leq C(n,p,K,\Omega) (1-s)^{1/p} \alpha_s(p)^{1/p} [\bu]_{\cX^{s,p}(\Omega)}
    \end{equation*}
    and so
    \begin{equation*}
        \begin{split}
	    (1-s)^{-1/p}\vnorm{\bu - \bspsi_1}_{L^p(\Omega)} 
            &\leq C(n,p,\Omega,K) \sum_{j=1}^K \left( (1-s)^{-1/p} \vnorm{\bu-\bspsi_j}_{L^{p}(\Omega_j)} + \alpha_s(p)^{1/p} [\bu]_{\cX^{s,p}(\Omega)} \right) \\
            &\leq C(n,p,\Omega) \alpha_s(p)^{1/p} [\bu]_{\cX^{s,p}(\Omega)}.
        \end{split}
	\end{equation*}
\end{proof}

\begin{proof}[Proof of Theorem \ref{thm:mainKorn1st:Endpoints} when $p = \infty$]
    In this setting we do not use the result of Theorem \ref{thm:KornBallStrong:Linfty}, but rather the techniques of the proofs leading to it. First, a reading of the proof of Lemma \ref{lma:Unscaled:Infty} shows that its conclusions remain true if the condition $B(\bx_0,r/2) \subset G \subset B(\bx_0,2r)$ is replaced with $B(\bx_0,\kappa r) \subset B(\bx_0,\Lambda r)$ for fixed $\kappa < 1 < \Lambda$ (the constant $C$ may additionally depend on $\kappa$ and $\Lambda$). Second, we show the following:
    \begin{equation}\label{eq:Lipschitz:ZoomIn}
        \begin{gathered}
            \text{There exists a constant } \kappa = \kappa(\Omega) > 0 \text{ such that, } \forall \bx_0 \in \Omega \text{ and } r \in (0,2\diam\Omega),\text{the set } \\
            B_{\bx_0,r} := \Omega \cap B(\bx_0,r) \text{ contains a ball of radius } \kappa r. \text{ Moreover, } B(\wt{\bx},\kappa r) \subset B_{\bx_0,r} \subset B(\wt{\bx},2 r).
        \end{gathered}
    \end{equation}
    Assuming \eqref{eq:Lipschitz:ZoomIn}, we prove the theorem. With Theorem \ref{thm:Scaled:Infty} in mind, denote $\bar{d} = 2 \diam(\Omega)$. For $\bx \in \Omega$ and $r \in (0,2d)$, thanks to \eqref{eq:Lipschitz:ZoomIn} we can use the modification of Lemma \ref{lma:Unscaled:Infty} to conclude that there exists $\bspsi_{\bx,r} \in \cR$ satisfying
    \begin{equation*}
        \vnorm{\bu - \bspsi_{\bx,r}}_{L^\infty(B_{\bx_0,r})} \leq C(n,\Omega) [\bu]_{\cK^s(\Omega)} r^s.
    \end{equation*}
    For $k \in \bbN_0$ abbreviate $\bspsi_k = \bspsi_{\bx_0,2^{-k}\bar{d}}$, with $\bspsi_k(\bx) = \bA_k \bx + \bfb_k$. Just as in the proof of Theorem \ref{thm:Scaled:Infty}, $|\bA_0 - \bA_k| \leq \frac{C(n,\Omega)}{1-2^{s-1}} [\bu]_{\cK^s(\Omega)} 2^{-k(s-1)}$.

    Now, let $\by \in \Omega$, and let $k \in \bbN_0$ such that $|\bx-\by| \in (2^{-k-1} \bar{d}, 2^{-k} \bar{d})$. Then $\bx$, $\by \in B_{\bx,2^{-k} \bar{d}}$, and just as in the conclusion of the proof of Theorem \ref{thm:Scaled:Infty}
    \begin{equation*}
        |\bu(\bx) - \bu(\by) - \bA_0(\bx-\by)| \leq 2 C(n,\Omega) [\bu]_{\cK^s(\Omega)} 2^{-ks} + |\bA_k - \bA_0| |\bx-\by| \leq \frac{C}{1-s} [\bu]_{\cK^s(\Omega)} |\bx-\by|^s.
    \end{equation*}
    An argument similar to that of the proof of Theorem \ref{thm:Korns1stOnBall:Linfty} shows that we can replace the rigid motion $\bA_0 \bx + \bfb_0$ with $\Pi_{B(\bx_0,r_0)}[\bu]$, where $B(\bx_0,r_0) \subset \Omega$ is any fixed Euclidean ball. Thus the theorem is proved.

    Now we must go back and prove \eqref{eq:Lipschitz:ZoomIn}. Let $\bx_0 \in \Omega$ and let $r \in (0,2 \bar{d})$. If $\dist(\bx_0,\p \Omega) \geq r/4$, choose $\wt{\bx} = \bx_0$ and $\kappa = 1/8$. If $\dist(\bx_0,\p \Omega) < r/4$, define $\bp(\bx_0) \in \p \Omega$ to be a nearest point projection. Then since $\Omega$ is a Lipschitz domain, it satisfies the uniform interior cone condition; there exists $\beta = \beta(\Omega) \in (0,1)$ and $\bsomega(\bx_0) \in \bbS^{n-1}$ such that $K_{\bsomega,\beta}(\bp(\bx_0),r/2)$, the open cone centered at $\bp(\bx_0)$ of radius $r/2$ with axis in the direction of $\bsomega$ and aperture width $2 \beta$, is contained in $\Omega \cap B(\bx_0,r)$ and moreover contains $\bx_0$. The parameter $\beta$ is independent of $\bx_0$ and $r$. Then for $\wt{\bx} = \bp(\bx_0) + \frac{r}{2(1+\sqrt{1-\beta^2})} \bsomega$ and $\kappa = \frac{\sqrt{1-\beta^2}}{\sqrt{1-\beta^2} +1}$ we have $B(\wt{\bx},\kappa r) \subset K_{\bsomega,\beta}(\bp(\bx_0),r/2)$ by elementary geometric observations and the triangle inequality. This choice of $\wt{\bx}$ and $\kappa$ ensures all the desired properties hold in the case $\dist(\bx_0,\p \Omega) < r/4$. Finally, to ensure \eqref{eq:Lipschitz:ZoomIn} holds for all $\bx_0 \in \Omega$ and $r \in (0,2 \bar{d})$, define $\kappa = \frac{1}{8} \frac{\sqrt{1-\beta^2}}{\sqrt{1-\beta^2} +1}$.
\end{proof}

We close this section by proving Theorem \ref{thm:mainKorn2nd}:

\begin{proof}[Proof of Theorem \ref{thm:mainKorn2nd}]
    We use \eqref{eq:mainKorn1stCase2:Proj} and \eqref{eq:mainKorn1stCase2:Proj:Endpoints}.
    For $p \in [1,\infty)$, by the triangle inequality, an estimate similar to that in \eqref{eq:NormCompEquiv}, 
    and the continuity assumption on $\Pi$, 
    \begin{equation}
        \begin{split}
            [\bu]_{W^{s,p}(\Omega)} &\leq [\bu-\Pi[\bu]]_{W^{s,p}(\Omega)} + [\Pi[\bu]]_{W^{s,p}(\Omega)} \\ &\leq C \alpha_s(p)^{1/p} [\bu]_{\mathcal{X}^{s,p}(\Omega)} + \frac{C}{(1-s)^{1/p}} \|\Pi[\bu]\|_{\mathrm{aff}} \\
            &\leq C \left( \alpha_s(p)^{1/p} [\bu]_{\mathcal{X}^{s,p}(\Omega)} + (1-s)^{-1/p} \|\bu\|_{L^p(\Omega)} \right).
        \end{split}
    \end{equation}    
    Thus \eqref{eq:mainKorn2nd} is proved for $p < \infty$. The argument is similar for $p = \infty$, with \eqref{eq:NormCompEquiv:Infty} used in place of \eqref{eq:NormCompEquiv}.
    In either case, adding $\|\bu\|_{L^p(\Omega)}$ to both sides yields $\|\bu\|_{W^{s,p}(\Omega)} \leq C \|\bu\|_{\mathcal{X}^{s,p}(\Omega)}$. The reverse inequality $\|\bu\|_{\mathcal{X}^{s,p}(\Omega)} \leq \|\bu\|_{W^{s,p}(\Omega)}$ is trivial since $\left| (\bu(\by)-\bu(\bx)) \cdot \frac{\by-\bx}{|\by-\bx|} \right| \leq |\bu(\by)-\bu(\bx)|$. Thus $\mathcal{X}^{s,p}(\Omega) = W^{s,p}(\Omega;\mathbb{R}^n)$.
\end{proof}

Finally, we remark on the consistency of the inequalities when $p \in (1,\infty)$ when $s \to 1$.
Using nonlocal characterizations of Sobolev functions and vector fields \cite{BourgainBrezisMironescu2001, Mengesha2012}, for any bounded Lipschitz domain $\Omega$ and for any $\bu \in W^{1,p}(\Omega;\bbR^n)$
\begin{equation*}
    (1-s)[\bu]_{W^{s,p}(\Omega)}^p \to \bar{C} \vnorm{\grad \bu}_{L^p(\Omega)}^p, \qquad \text{ and } (1-s)[\bu]_{\cX^{s,p}(\Omega)}^p \to \bar{C} \vnorm{\cL(e(\bu))}_{L^1(\Omega)}^p \qquad \text{ as } s \to 1,
\end{equation*}
where $\bar{C}(n,p)$ is a constant and where $\cL : \bbR^{n \times n} \to \bbR$ is a functional comparable to $|\cdot|^p$ on $\bbR^{n \times n}$.
Therefore, \eqref{eq:mainKorn1stCase2:Proj} becomes $\vnorm{\bu - \Pi[\bu]}_{W^{1,p}(\Omega)}^p \leq C \vnorm{e(\bu)}_{L^p(\Omega)}^p$
and \eqref{eq:mainKorn2nd} becomes $\vnorm{\grad \bu}_{L^p(\Omega)}^p \leq C \vnorm{e(\bu)}_{L^p(\Omega)}^p + C \vnorm{\bu}_{L^p(\Omega)}^p$ as $s \to 1$, for any $\bu \in W^{1,p}(\Omega;\bbR^n)$.

By minimizing $\vnorm{\bu - \bspsi}_{L^2(\Omega)}$ and $[\bu - \bspsi]_{W^{s,2}(\Omega)}$ over $\cR$, it is possible to show that $\Pi[\bu]$ can be replaced in \eqref{eq:mainKorn1stCase2:Proj} when $p = 2$ with the infinitesimal rigid motion
\begin{equation*}
    \bspsi(\bx) = \frac{1}{2}(\bA[\bu]-\bA[\bu]^T) \bx + \bfb[\bu],
\end{equation*}
where
\begin{equation*}
     \bA[\bu] := \left( \int_\Omega \int_\Omega \frac{ (\bx-\by) \otimes (\bx-\by) }{ |\bx-\by|^{n+2s} } \, \rmd \by \, \rmd \bx \right)^{-1} \int_\Omega \int_\Omega \frac{ (\bu(\bx)-\bu(\by)) \otimes (\bx-\by) }{ |\bx-\by|^{n+2s} } \, \rmd \by \, \rmd \bx
\end{equation*}
and
\begin{equation*}
    \bfb[\bu] = \fint_{\Omega} \bu(\bx) \, \rmd \bx - \frac{\bA[\bu]-\bA[\bu]^T}{2} \fint_{\Omega} \bx \, \rmd \bx.
\end{equation*}
In fact, this replacement can be made for any $p$; the operator $\Psi$ defined by $\Psi[\bu] : \bA[\bu]\bx + \bfb[\bu]$ is continuous from $W^{s,p}(\Omega;\bbR^n)$ to $\cR$, and so \eqref{eq:mainKorn2nd} will imply \eqref{eq:mainKorn1stCase2:Proj} for $\bu - \Psi[\bu]$ via a compactness argument.

\section{Proofs of the constrained Korn inequalities}\label{sec:subspaces}

In this section we prove Corollary \ref{cor:mainKorn1st:Constrained} and Propositions \ref{prop:ZeroTraceSubspace}--\ref{prop:LionsMagenesKorn}.

\begin{proof}[Proof of Corollary \ref{cor:mainKorn1st:Constrained}]
    To prove \eqref{eq:mainKorn1stCase1}, suppose by contradiction that there exists a sequence $\{\bu_k\} \subset \mathring{\cX}^{s,p}(\Omega) \cap W^{s,p}(\Omega)$ that satisfies
    \begin{equation}\label{eq:korn1stcase1:pfbound2}
        \vnorm{ \bu_k }_{W^{s,p}(\Omega)} = 1 \qquad \text{ and } \vnorm{ \bu_k }_{\cX^{s,p}(\Omega) } \leq 1/k, \qquad k \in \bbN.
    \end{equation}
    By a fractional analog  of the Rellich-Kondrachov compactness theorem \cite{DiNezza}, there exists a subsequence (not relabeled) $\bu_k \to \bu$ strongly in $L^p(\Omega;\mathbb{R}^n)$ for some $\bu \in W^{s,p}(\Omega) $. 
    We then have by Korn's second inequality \eqref{thm:mainKorn2nd} in   Theorem \ref{thm:mainKorn2nd} and by \eqref{eq:korn1stcase1:pfbound2} that
    \begin{equation*}
        \begin{split}
            \vnorm{\bu_k - \bu_m}_{W^{s,p}(\Omega)} &\leq C \left( \vnorm{\bu_k - \bu_m}_{L^p(\Omega)} + \vnorm{\bu_k - \bu_m}_{\cX^{s,p}(\Omega)} \right) \\
            &\leq C(1/k + 1/m + \vnorm{\bu_k - \bu_m}_{\cX^{s,p}(\Omega)}).
        \end{split}
    \end{equation*}
    Therefore the sequence $\{ \bu_k \}$ is Cauchy and thus convergent (to $\bu$) in $W^{s,p}(\Omega)$.
    So by \eqref{eq:korn1stcase1:pfbound2}
    \begin{equation*}
        [\bu]_{\cX^{s,p}(\Omega)} \leq [\bu_k]_{\cX^{s,p}(\Omega)} + [\bu_k - \bu]_{\cX^{s,p}(\Omega)} \leq 1/k + [\bu_k - \bu]_{W^{s,p}(\Omega)} \to 0
    \end{equation*}
    as $k \to \infty$. 
    Thus $[\bu]_{\mathcal{X}^{s,p}(\Omega)} = 0$, so $\bu \in \mathcal{R}$ \cite{Mengesha2012, Temam-Miran}.
    But then since $\bu \in \mathring{\cX}^{s,p}(\Omega)$ it follows that $\bu = 0$, a contradiction since $1 = \vnorm{\bu}_{W^{s,p}(\Omega)} = \lim\limits_{k \to \infty} \vnorm{\bu_k}_{W^{s,p}(\Omega)}$.
\end{proof}

\begin{proof}[Proof of Proposition \ref{prop:ZeroTraceSubspace}]
    For $sp > 1$, vector fields in $W^{s,p}(\Omega;\mathbb{R}^n)$ possess well-defined boundary traces in $L^p(\partial\Omega;\mathbb{R}^n)$. The only rigid motion $\bspsi(\bx) = \mathbf{A}\bx + \mathbf{b} \in \mathcal{R}$ satisfying $\bspsi|_\Gamma = \mathbf{0}$ on a set $\Gamma$ of positive surface measure is $\bspsi= \mathbf{0}$. Thus $\mathcal{X}^{s,p}_{0,\Gamma}(\Omega) \cap \mathcal{R} = \{\mathbf{0}\}$, so Corollary \ref{cor:mainKorn1st:Constrained} applies.
    For $sp \leq 1$, $W^{s,p}_0(\Omega) = W^{s,p}(\Omega)$ \cite{DiNezza, leoni2023first}, so non-trivial rigid motions belong to the space, causing \eqref{eq:mainKorn1stCase1} to fail.
\end{proof}

\begin{proof}[Proof of Proposition \ref{prop:PeridynamicCollar}]
    If $\bspsi(\bx) = \mathbf{A}\bx + \mathbf{b} \in \mathcal{R}$ vanishes a.e. on a set $\widetilde{\Omega}$ of positive Lebesgue measure, then $\mathbf{A} = \mathbf{0}$ and $\mathbf{b} = \mathbf{0}$. Thus $\widetilde{\mathcal{X}}^{s,p}_{\widetilde{\Omega}}(\Omega) \cap \mathcal{R} = \{\mathbf{0}\}$, so Corollary \ref{cor:mainKorn1st:Constrained} applies.
\end{proof}

\begin{proof}[Proof of Proposition \ref{prop:FunctionalConstraint}]
    If $\bspsi\in \dot{\mathcal{X}}^{s,p}(\Omega) \cap \mathcal{R}$, then $P[\bspsi] = \bspsi= \mathbf{0}$. Thus Corollary \ref{cor:mainKorn1st:Constrained} applies.
\end{proof}

\begin{proof}[Proof of Proposition \ref{prop:LionsMagenesKorn}]
    Let $\bu \in \mathcal{X}^{s,1/s}_{00}(\Omega)$ and let $\widetilde{\bu}$ be its extension by zero to $\mathbb{R}^n$. Applying the fractional Korn inequality on $\mathbb{R}^n$ \cite{MScott2018Korn}, we obtain $[\widetilde{\bu}]_{W^{s,1/s}(\mathbb{R}^n)} \leq C [\widetilde{\bu}]_{\mathcal{X}^{s,1/s}(\mathbb{R}^n)}$. Since $\|\bu\|_{W^{s,1/s}_{00}(\Omega)}^{1/s} \leq [\widetilde{\bu}]_{W^{s,1/s}(\mathbb{R}^n)}^{1/s}$ and $[\widetilde{\bu}]_{\mathcal{X}^{s,1/s}(\mathbb{R}^n)}^{1/s} \leq C \|\bu\|_{\mathcal{X}^{s,1/s}_{00}(\Omega)}^{1/s}$, the result follows.
\end{proof}

We close this section by proving the following inequality for vector fields defined on $\mathbb{R}^d$.  
\begin{corollary}\label{cor:Wegner-app}
    There exists a constant $C = C(n,p) > 0$ such that for all $\bu \in L^p_{\mathrm{loc}}(\mathbb{R}^n;\mathbb{R}^n)$,
    \begin{equation}\label{eq:Wegner-app}
        \sup_{B \subset \mathbb{R}^n} \inf_{\bspsi \in \mathcal{R}} [\bu - \bspsi]_{W^{s,p}(B)} \leq C [\bu]_{\mathcal{X}^{s,p}(\mathbb{R}^n)},
    \end{equation}
    where the supremum is taken over all Euclidean balls $B \subset \mathbb{R}^n$.
\end{corollary}

\begin{proof}
    Let $B = B(\bx_0, R) \subset \mathbb{R}^n$ be an arbitrary ball. By Theorem \ref{thm:KornBallStrong}, the projection operator $\Pi_B[\bu] \in \mathcal{R}$ satisfies
    \begin{equation}
        [\bu - \Pi_B[\bu]]_{W^{s,p}(B)} \leq C(n,p) [\bu]_{\mathcal{X}^{s,p}(B)},
    \end{equation}
    where the constant $C(n,p) > 0$ is independent of the center $\bx_0$ and radius $R$. Since $\Pi_B[\bu] \in \mathcal{R}$ and $[\bu]_{\mathcal{X}^{s,p}(B)} \leq [\bu]_{\mathcal{X}^{s,p}(\mathbb{R}^n)}$, we bound the infimum over $\mathcal{R}$ by choosing $\bspsi= \Pi_B[\bu]$:
    \begin{equation}
        \inf_{\bspsi\in \mathcal{R}} [\bu - \bspsi]_{W^{s,p}(B)} \leq [\bu - \Pi_B[\bu]]_{W^{s,p}(B)} \leq C(n,p) [\bu]_{\mathcal{X}^{s,p}(B)} \leq C(n,p) [\bu]_{\mathcal{X}^{s,p}(\mathbb{R}^n)}.
    \end{equation}
    Taking the supremum over all balls $B \subset \mathbb{R}^n$ completes the proof.
\end{proof}

\appendix
\section{Some auxilary results }
\begin{lemma}\label{lma:PositiveFunctional}
    Let $p \in [1,\infty)$. Define $\Phi : \bbS^{n-1} \times \overline{B(0,3)}  \to [0,\infty)$ via
    \begin{equation*}
        \Phi(\bsomega,\bsxi) := \int_{B(0,1)} \bb1_{ \{ |\by-\bsxi| \geq 1/8 \} } \left| \bsomega \cdot \frac{\bsxi - \by}{|\bsxi-\by|} \right|^p \, \rmd \by,
    \end{equation*}
    where $\mathds{1}_E$ denotes the indicator function on a measurable set $E$. 
    Then
    \begin{equation}\label{eq:UnscaledBdyCoercivity}
        \min_{ \bbS^{n-1} \times \overline{B(0,3)}} \Phi :=  C_1 > 0.
    \end{equation}
\end{lemma}

\begin{proof}
    By dominated convergence, $\Phi$ is a continuous function on a compact set, hence attains its minimum. $\Phi$ is nonnegative, so we argue by contradiction; assume that its minimum is $0$ attained for some $(\bsomega,\bsxi)$. Then 
    \begin{equation*}
        |\bsomega \cdot (\bsxi - \by)| = 0 \qquad \forall \by \in B(0,1) \cap S_{\bsxi},
    \end{equation*}
    where $S_{\bsxi} := \{ \by \in \bbR^n \, : \, |\by-\bsxi| \geq 1/8\}$. Therefore $B(0,1) \cap S_{\bsxi}$ must be contained in an $n-1$-dimensional hyperplane.
    However, $B(0,1) \cap S_{\bsxi}$ contains a Euclidean ball for all $\bsxi \in \overline{B(0,3)}$ (for instance, $B(0,1/8) \subset B(0,1) \cap S_{\bsxi}$ whenever $|\bsxi| > 1/4$, and $B(\frac{7}{8} \ba, \frac{1}{8}) \subset B(0,1) \cap S_{\bsxi}$ whenever $|\bsxi| \leq 1/4$, for any fixed unit vector $\ba$). This is a contradiction, and therefore $\Phi > 0$ on $\bbS^{n-1} \times \overline{B(0,3)}$.
\end{proof}

\end{document}